\documentclass[11pt]{article}
\usepackage{amsmath,amssymb,a4wide,amsthm,color,graphicx,verbatim,algorithmic,algorithm,caption,enumerate,multirow}
\usepackage{epstopdf}
\usepackage{amsfonts}

\newtheorem{theorem}{Theorem}[section]

\newtheorem{rem}[theorem]{Remark}
\newtheorem{notation}[theorem]{Notation}
\newenvironment{remark}{\begin{rem} \em }{\em \end{rem}}

\newtheorem{proposition}[theorem]{Proposition}
\newtheorem{ex}[theorem]{Example}
\newenvironment{example}{\begin{ex} \em }{\em \end{ex}}
\newtheorem{definition}[theorem]{Definition}
\newtheorem{assumption}[theorem]{Assumption}

\newcommand{\norm}  [1]{\ensuremath{\left  \|       #1  \right \|       }}

\newcommand{\abs  } [1]{\ensuremath{\left  |       #1  \right |        }}

\newcommand{\absgg} [1]{\ensuremath{\biggl |       #1  \biggr |        }}

\newcommand{\eps}{\ensuremath{\varepsilon}}

\newcommand{\st} {\ensuremath{|\;}}

\usepackage{color}
\definecolor{orange}{RGB}{250, 54, 0}
\definecolor{purple}{rgb}{0.75, 0.0, 1.0}

\newcommand{\SHP} {{\rm SHP \,}}
\newcommand{\SubHP} {{\rm SubHP \,}}

\newcommand{\cl}  {{\rm cl  \,}}
\newcommand{\bd}  {{\rm bd \,}}

\newcommand{\Int} {{\rm int \,}}
\newcommand{\conv}  {{\rm conv \,}}

\newcommand{\dom}{{\rm dom\,}}

\newcommand{\E}{\mathcal{E}}

\newcommand{\R}{\mathbb{R}}

\renewcommand{\E}{\mathbb E}
\renewcommand{\P}{\mathbb P}

\newcommand{\Max}{{\rm Max}\,}

\newcommand{\Maxk}{{\rm Max}_K\,}

\newcommand{\wMax}{{\rm wMax}\,}
\newcommand{\nwMax}{{\rm (w)Max}\,}

\newcommand{\wMaxk}{{\rm wMax}_K\,}

\newcommand{\Mink}{{\rm Min}_K\,}

\newcommand{\wMin}{{\rm wMin}\,}
\newcommand{\nwMin}{{\rm (w)Min}\,}

\newcommand{\wMink}{{\rm wMin}_K\,}
\newcommand{\nwMink}{{\rm (w)Min}_K\,}

\newcommand{\C} {{\rm C \,}}
\newcommand{\Cupp} {{\rm C^{up} \,}}
\newcommand{\Clow} {{\rm C^{lo} \,}}

\newcommand{\Cs} {{\rm C^{s} \,}}
\newcommand{\Cw} {{\rm C^{w} \,}}
\newcommand{\cs} {{\rm c^{s} \,}}
\newcommand{\cw} {{\rm c^{w} \,}}

\author{Birgit Rudloff \thanks{Vienna University of Economics and Business, Institute for Statistics and Mathematics, Vienna 1020, Austria, birgit.rudloff@wu.ac.at.} 
\and Firdevs Ulus \thanks{Bilkent University, Department of Industrial Engineering, Ankara, 06800 Turkey, firdevs@bilkent.edu.tr}}

\title{Certainty Equivalent and Utility Indifference Pricing for Incomplete Preferences via Convex Vector Optimization }

\date{\today}

\begin{document}
\maketitle

\begin{abstract} 
For incomplete preference relations that are represented by multiple priors and/or multiple -possibly multivariate- utility functions, we define a certainty equivalent as well as the utility 
indifference price bounds as set-valued functions of the claim. Furthermore, we motivate and introduce the notion of a weak and a strong certainty equivalent. We will show that our definitions contain as special cases some definitions found in the literature so far on complete or special incomplete preferences. We prove monotonicity and convexity properties of utility buy and sell prices that hold in total analogy to the properties of the scalar indifference prices for complete preferences. We show how the (weak and strong) set-valued certainty equivalent as well as the indifference price bounds can be computed or approximated by solving convex vector optimization problems. Numerical examples and their economic interpretations are given for the univariate as well as for the multivariate case.
\medskip

\noindent
{\bf Keywords:} utility maximization, indifference price bounds, certainty equivalent, incomplete preferences, convex vector optimization.

\medskip

\noindent
{\bf JEL Classification:} D81, C61, G13

\end{abstract}

\section{Introduction}
\label{sect:Intro}

The certainty equivalent of a random payoff is a guaranteed return that a decision maker would accept now as it is equally desirable as the uncertain return that will be received in the future. Indifference pricing can be seen as a similar concept adapted to a dynamic setting. It plays an important role in pricing in incomplete markets as it typically yields a more narrow pricing interval compared to the often very wide no-arbitrage pricing interval, see for instance~\cite{henderson}.

The certainty equivalent and utility indifference pricing are well studied for complete preference relations that can be represented by a single univariate utility function and there are also some extensions for complete preferences represented by a single multivariate utility function. However, the completeness assumption of the preference relation is restrictive as it ignores the typical `indecisiveness' that individuals face. This concern was stated already by von Neumann and Morgenstern in their 1947 paper~\cite{vonNeuman1947} as ``It is conceivable -and may even in a way be more realistic- to allow for cases where the individual is neither able to state which of two alternatives he prefers nor that they are equally desirable." As Aumann~\cite{Aumann1962} and many researchers agreed afterwards, it is natural and indeed more realistic to exclude the completeness axiom when considering preference relations. 

We introduce a certainty equivalent and indifference buy (and sell) price concepts for underlying preferences that are not necessarily complete. In particular, let us consider a probability space $(\Omega,\mathcal{F},\P)$, the set of all $\mathcal{F}$-measurable $\R^d$-valued random vectors $L^0(\mathcal{F},\R^d)$ and a preorder $\succsim$ on $L^0(\mathcal{F},\R^d)$.

Note that if $d=1$ and the preference relation is complete, then the certainty equivalent of a random amount $Z\in L^0(\mathcal{F},\R)$ can be described as the deterministic amount, denoted by $C(Z)$, satisfying $Z \sim C(Z)$.  Under standard monotonicity and continuity assumptions, $C(Z)$ exists and it is unique, hence well-defined. When $d>1$, the deterministic amount indifferent to $Z\in L^0(\mathcal{F},\R^d)$ may not be unique. A natural way to deal with this problem is to consider the set of all such certain amounts. In other words, {one can} define the certainty equivalent as $$C(Z) := \{c\in\R^d \st c \sim Z\}$$ for complete as well as for incomplete preference relations. 
However, it is restrictive in the sense that whenever the preferences are incomplete, the certainty equivalent may be an empty set, {and thus}, it may fail to capture all the information that it captures in the complete preference case. Therefore, we {propose to} consider {also} the set of certain amounts $c$ for which the decision maker prefers $c$ to $Z$; and symmetrically, the set of certain amounts $c$ for which the decision maker prefers $Z$ to $c$, that is, we consider the sets $\{c\in\R^d \st c \succsim Z\}$ and $\{c\in\R^d \st c\precsim Z\}$. Clearly, the certainty equivalent is the intersection of these two sets, {but} whenever it is empty, the two sets above would {still} provide the full information to the decision maker.

Mimicking the definition for the certainty equivalent, the indifference buy (sell) price of a claim $C$ {could} be defined as the set of prices $p$ such that the decision maker is indifferent between buying (selling) the claim at price $p$ and not buying (selling) it at all. Since such price may not exist when the preference relation is incomplete, we instead consider the set of all prices for which the decision maker has a preference of buying/selling the claim over taking no action, namely, the \emph{set valued buy (sell) prices}.

In order to analyze these set-valued concepts in detail and in order to be able to compute these for practical reasons, we assume for the rest of the paper that the underlying probability space is finite and the incomplete preference relation accepts a representation. Note that the incompleteness of preferences of a decision maker may stem from different reasons. First, certain outcomes might be incomparable for the decision maker. A simple example is the case where the decision maker is a committee instead of an individual. Ok~\cite{ok2000}, and Dubra, Maccheroni, and Ok~\cite{ok2004} suggested vector-valued utility representations in order to deal with such preferences. Secondly, even though the decision maker has a complete preference over the set of all outcomes, the incompleteness may occur because of the decision maker's indecisiveness on the likelihood of the states of the world. This is known as Bewley's model of Knightian uncertainty~\cite{bewley}.  

In 2006, Nau \cite{Nau} considered preferences which are allowed to be incomplete in both senses and studied the representation of them. Indeed, allowing both types of incompleteness leads to a representation of preferences by a set of probability measures paired with utility functions. The representations of incomplete preferences are further studied for instance by Ok, Ortoleva and Riella \cite{ok2012} and Galaabaatar and Karni \cite{galaabaatar2013}. In the former, \textit{single-prior expected multi-utility representation} and \textit{multi-prior expected single-utility representation}; whereas in the latter \textit{multi-prior expected multi-utility representation} of incomplete preferences are axiomatized. In both papers the state space is assumed to be finite. In~\cite{galaabaatar2013}, the outcome (prize) space is also assumed to be finite, whereas in~\cite{ok2012} it is a compact metric space.

In this paper, the state space is assumed to be finite and the outcome space is $\R^d$ for $d \geq 1$. We consider preferences on $\R^d$-valued random vectors where the utility functions are multivariate for $d>1$ and where the preference relations are represented by a set of probability measures and a set of utility functions as in \cite{galaabaatar2013}. As stated in \cite{Nau}, this representation ``preserves the traditional separation of information about beliefs from information about values", which ``arises naturally when imprecise probabilities and utilities are assessed independently, as they often are in practice". Note however that it is also possible to consider a set of probability-utility pairs as in \cite{Nau} without changing the main results of this paper. 

As stated before, the certainty equivalent set can be empty and in order to capture the full information one can consider the set of better/worse values instead of considering the indifferent ones. Indeed, in the special case $d=1$ and an incomplete preference relation admitting a single-prior expected multi-utility representation, Armbruster and Delage~\cite{CEworst} defined a `strong certainty equivalent'. In a symmetrical way, it is also possible to consider a `weak certainty equivalent'.
A direct extension of this definition to the case $d>1$ is not straightforward, but the construction of the set of all better/worse values described above allows us to define a set-valued strong certainty equivalent as well as a set-valued weak certainty equivalent also in this case. This definition reduces indeed to the usual definition whenever $d=1$ and the preference relation is complete, as well as to the definition of~\cite{CEworst} when $d=1$ and the  incomplete preference relation admits a single-prior expected multi-utility representation. Properties, interpretations and examples will be given for the case $d>1$ as well as for $d=1$.

In the literature, there are different certainty equivalent concepts for $d>1$ when a complete preference relation admitting a single-prior single-utility representation with a multivariate utility function is considered, see the survey~\cite{pagani}. In~\cite{pagani}, {it is stated that} no vector-valued or set-valued certainty equivalent concept has been introduced for multivariate utility functions so far. However, a set-valued certainty equivalent definition is provided in~\cite{juniorproj} for a {multi-asset game setting. In particular, a set-valued utility function, which depends on the exchange structure of the multi-asset model and a vector valued utility function, as well as a set valued certainty equivalent for this particular setting are introduced in~\cite{juniorproj}. A parametric representation of the certainty equivalent of a particular game, where a component-wise vector valued utility function is used to define the set valued utility, is computed analytically. Here, we provide a set-valued definition of a certainty equivalent for a much more general setting.} 

In addition to the certainty equivalent, we study utility indifference price bounds under an incomplete preference that admits a multi-prior expected multi-utility representation where utility functions are allowed to be multivariate. This is done by first considering the set-valued buy and sell prices and then defining the utility indifference price bound as the boundaries of these sets. We show that these definitions of buy and sell price bounds {have intuitive interpretations and they} recover the complete preference case when the utility function is univariate. Moreover, we will prove that the set-valued buy and sell prices satisfy some monotonicity and convexity properties in total analogy to the properties of the scalar indifference prices for complete preferences.

Utility indifference buy and sell prices for a complete preference relation represented by a multivariate utility function under proportional transacation costs have been studied by Benedetti and Campi in~\cite{benedetti2012}. Accordingly, the buy and sell prices, $p^b_j, p^s_j$ are defined in terms of a single currency $j \in \{1,\ldots,d\}$. It has been shown in~\cite{benedetti2012} that $p^b_j, p^s_j$ are well defined, they exists uniquely under the conical market model. We show that the set-valued prices contain the scalar prices defined in~\cite{benedetti2012}. In particular, $p^b_je_j$ and $p^s_je_j$ for all $j \in \{1,\ldots,d\}$, where $e_j$ is the unit vector with $j^{th}$ component being 1, are on the boundary of the set-valued buy and sell prices, respectively. They correspond to the indifference prices, if one has initial capital in one of the $d$ currencies only. In contrast, the set-valued indifference price bounds defined here also allow for an initial portfolio in the $d$ currencies and allow also for incomplete preference relations.

Recently, Hamel and Wang~\cite{Umax_setopt} have considered the utility maximization problem under proportional transaction costs, where the market is modeled by solvency cones and the preferences are represented by component-wise utility functions. The motivation behind this is that independent from holdings in the other {assets}, the investor has a scalar utility function for each of them. Clearly, this is a special type of vector-valued utility function. We consider this set up as a special case and discuss the certainty equivalent and indifference price bounds concepts introduced here under this set up.

For practical reasons, it is important that the set valued certainty equivalent and the buy and sell price bounds introduced here can be computed as well. Indeed, we show that the computations require solving convex vector optimization problems (CVOPs).
	
In the literature, there are several algorithms and methods to `solve'  some specific subclasses of CVOPs, see the survey paper by Ruzika and Wiecek \cite{surveyruzika}. For more general problems, Ehrgott, Shao and Sch\"obel \cite{ehrgott} developed an approximation algorithm and more recently, L\"{o}hne, Rudloff and Ulus \cite{cvop} generalized Benson's algorithm (see \cite{benson}) and proposed two algorithms to generate approximations to the set of all efficient values in the objective space. One of the algorithms is the extension of the one proposed in \cite{ehrgott} while the second one is the `geometric dual' of it.
	
We show that as long as not empty, the set-valued (strong/weak) certainty equivalent can be computed by solving CVOPs. Moreover, as in the complete preference case, the computations of the buy and sell price bounds require solving the utility maximization problem, which is naturally modeled as a CVOP in our setting. We use the algorithms provided in \cite{cvop} to approximately solve the utility maximization problem. We show that it is possible to compute inner and outer `approximations' to the set-valued buy and sell prices by solving CVOPs where the solution of the utility maximization problem is taken as an input. As in the complete preference case, solving the optimization problem(s) also yields the hedge positions. In the example section, we {illustrate} the economic meaning of our definitions of the set-valued certainty equivalent as well as the set-valued buy and sell prices. 

The organization of the paper is as follows. In Section~\ref{sect:prelim}, we introduce the notation that is used throughout this paper and review some basic results on classical utility indifference pricing and on representations of incomplete preference relations. 
In Section~\ref{sect:CE}, we introduce the set-valued definition of the certainty equivalent {as well as the strong and weak version of it}. Set-valued buy and sell prices as well as indifference price bounds are introduced in Section~\ref{sect:indifference}. In this section, we also prove the properties of set-valued buy and sell prices. The computations of these set valued quantities are explained in Section~\ref{sect:computations}. The last section provide some special cases and numerical examples. In Section~\ref{subsect:examples1}, we set $d=1$ and consider univariate utility functions, while in Section~\ref{sect:Conical}, we consider the conical market model for $d>1$.

\section{Preliminaries}
\label{sect:prelim}
In the following we introduce some basic notions regarding order relations and convex vector optimization problems. Then, we review the basic definition of indifference pricing in the classical expected utility theory. Finally, we recall the utility representations for incomplete preference relations that will be used here.
\subsection{Order Relations}
\label{subsect:PrelimOrders}
A convex cone $K \subseteq \R^q$ is said to be \emph{solid}, if it has a non-empty interior; \emph{pointed} if it does not contain any line; and \emph{non-trivial} if $\{0\} \subsetneq K \subsetneq \R^q$. A non-trivial convex pointed cone K defines a partial ordering $\leq_K$ on $\R^q$: $v \leq_K w$ if and only if $w - v \in K$; $v <_K w$ if and only if $w - v \in \Int K$; and $v \lneq_K w$ if and only if $w - v \in K\setminus\{0\}$.

Let $K\subseteq \R^q$ be a non-trivial convex pointed cone and $X \subseteq \R^d$ a convex set. A function $f:X \rightarrow \R^q$ is said to be \emph{$K$-convex} if $f(\alpha x+(1-\alpha)y) \leq_K \alpha f(x)+(1-\alpha)f(y)$ holds for all $x,y \in X$, $\alpha \in [0,1]$, and $K$-concave if $-f$ is $K$-convex., see e.g. \cite[Definition 6.1]{luc}. 

Let $A$ be a subset of $\R^q$. A point $y \in A$ is called a \emph{$K$-minimal element} of A if there exists no $x\in A\setminus\{y\}$ with $x \leq_K y$. If $K$ is solid, then a point $y \in A$ is called \emph{weakly $K$-minimal element} if there exists no $x\in A$ with $x <_K y$. The set of all (weakly) $K$-minimal elements of $A$ is denoted by $\nwMink(A)$. The set of (weakly) $K$-maximal elements is defined by  ${\rm (w)Max}_K\,(A):={\rm (w)Min}_{-K}\,(A)$. 

A convex pointed cone $K$ also defines two order relations on the power set of $\R^q$ as follows ({see for instance \cite{minimalset,Kuroiwa}}): For $A, B \subseteq \R^q$
\begin{equation}\label{setrelations}
A \preccurlyeq_K B :\iff B \subseteq A + K, \quad A \curlyeqprec_K B :\iff A \subseteq B - K.
\end{equation}
A set $A \subseteq \R^q$ is said to be an \emph{upper set} with respect to $K$ if  $A = A + K$, 
a \emph{lower set} with respect to $K$ if $A = A-K$. 
If $A$ is a closed upper set, then $\wMink A = \bd A$; similarly if $A$ is a closed lower set, then $\wMaxk A = \bd A$.

If $A$ and $B$ are closed upper sets {with respect to $K$}, then we have 
\begin{equation*}
\label{eq:comp_upper}
A \preccurlyeq_K B \iff \Mink A \preccurlyeq_K \Mink B \iff A \supseteq B;
\end{equation*}	
similarly, if $A$ and $B$ are closed lower sets {with respect to $K$}, then it is true that 
\begin{equation*}
\label{eq:comp_lower}
A \curlyeqprec_K B \iff \Maxk A \curlyeqprec_K \Maxk B \iff A \subseteq B.
\end{equation*} 

Whenever the ordering cone is $\R^q_+ = \{r\in \R^q \st r_i \geq 0, \: i=1,\ldots,q\}$, we write $\leq, \preccurlyeq, \curlyeqprec$  instead of $\leq_{\R^q_+}, \preccurlyeq_{\R^q_+}, \curlyeqprec_{\R^q_+}$; we say (weakly) minimal/maximal element instead of (weakly) $\R^q_+$-minimal / $\R^q_+$-maximal element, and denote the set of all such elements by $\nwMin(\cdot)$ / $\nwMax(\cdot)$. Moreover, an upper (lower) set with respect to $\R^q_+$ is simply said to be an upper (lower) set.

\subsection{Convex Vector Optimization Problems}
\label{subsect:prelimCVOP}

A convex vector optimization problem is to
\begin{align*} \label{(P)}
\text{minimize~} f(x) \text{~~~~with respect to~~} \leq_K \text{~~subject to~}  g(x) \leq_M 0, \tag{P}
\end{align*}
where $K\subseteq \R^q$, and $M\subseteq \R^m$ are non-trivial pointed convex ordering cones with nonempty interior, the vector-valued objective function $f(x)= (f_1,\ldots,f_q): \R^d \rightarrow \R^q$ is $K$-convex, and the constraint function $g = (g_1,\ldots,g_m) : \R^d \rightarrow \R^m$ is $M$-convex (see for example \cite{cvop,luc}). We denote the feasible region of~\eqref{(P)} by $\mathcal{X}:=\{x\in \R^d \st g(x)\leq_M 0\}$.

The set $\mathcal{P} := \cl (f(\mathcal{X})+K)$ is called the \emph{upper image}; it is an upper set with respect to $K$ and it satisfies $\wMin_K(\mathcal{P}) = \bd \mathcal{P}$.~\eqref{(P)} is said to be \emph{bounded} if the upper image is contained in $\{y\} + K$ for some $y \in \R^q$, that is, if  $\mathcal{P} \subseteq \{y\} + K$. A point $\bar{x}\in \mathcal{X}$ is a \emph{(weak) minimizer} for~\eqref{(P)} if $f(\bar{x})$ is a (weakly) $K$-minimal element of $f(\mathcal{X})$.

We consider a solution concept for CVOPs that relates a solution to an inner and an outer approximation of the upper image $\mathcal{P}$. Throughout $k \in \Int K$ is fixed.

\begin{definition}[\cite{cvop}]\label{defn:solnCVOP}
	For a bounded problem \eqref{(P)}, a nonempty finite set $\mathcal{\bar{X}}\subseteq \mathcal{X}$ is called a \emph{finite (weak) $\epsilon$-solution of \eqref{(P)}} if it consists of only (weak) minimizers and satisfies
	\begin{equation} \label{eq:weaksoln}
	\conv f(\mathcal{\bar{X}})+K-\epsilon \{k\} \supseteq \mathcal{P}.
	\end{equation}
\end{definition}

There are many different scalarization techniques for vector optimization problems. Two well-known ones will be used throughout.

The weighted sum scalarization of~\eqref{(P)} for $w \in \R^q$ is defined as the convex program
\begin{align*}\label{P_1(w)}
\text{minimize~} w^Tf(x) \text{~~subject to~}  g(x) \leq_M 0. \tag{$P_w$}
\end{align*}
The following proposition is well-known for CVOPs, see e.g.~\cite{jahn}. Here $K^+:=\{y\in\R^q \st \forall k\in K: k^Ty\geq 0 \}$ is the positive dual cone of $K$.
\begin{proposition}[\cite{jahn}]\label{prop:scalarization}
	An optimal solution of~\eqref{P_1(w)} for $w\in K^+\setminus\{0\}$ is a weak minimizer of~\eqref{(P)}. Moreover, if $\mathcal{X}\subseteq \R^d$ is a non-empty closed set, then for each weak minimizer $\bar{x}$ of~\eqref{(P)}, there exists $w \in K^+\setminus \{0\}$ such that $\bar{x}$ is an optimal solution to~\eqref{P_1(w)}.
\end{proposition}

The Pascoletti-Serafini~\cite{PascolettiSerafini} scalarization of~\eqref{(P)} for point $v \in \R^q$ and direction $d \in \R^q$ is defined as the convex program
\begin{align*}\label{P_2(v,d)}
\text{minimize~} \rho \text{~~subject to~}  g(x) \leq_M 0, \: f(x)-\rho d \leq_K v, \: \rho \in \R . \tag{$P_{(v,d)}$}
\end{align*}
\begin{proposition}[\cite{eichfelder}]\label{prop:scalarizationPS}
	{Let $(  \bar\rho,  \bar x)$ be an optimal solution of~\eqref{P_2(v,d)} for $v \in  \R^q, d \in K\setminus\{0\}$. Then $ \bar x$ is a weak minimizer of~\eqref{(P)}. }
\end{proposition}

A maximization problem with $K$-concave objective function $f(\cdot)$ is the negative of a CVOP with objective function $-f(\cdot)$. Clearly, \emph{the lower image} $\cl(f(\mathcal{X})-K)$ of a maximization problem is the negative of the upper image of the corresponding CVOP.

\begin{remark}\label{rem:innerouterappr}
	In~\cite{cvop}, L\"{o}hne, Rudloff and Ulus proposed primal and dual approximation algorithms to solve bounded CVOPs where ordering cones $K$ and $M$ are polyhedral. Both algorithms return finite weak $\epsilon$-solutions to~\eqref{(P)}. A weak $\epsilon$-solution $\bar{\mathcal{X}}$ to~\eqref{(P)} provides an inner and an outer approximation to the upper image $\mathcal{P}$ as
	$$\mathcal{P}_{\text{in}} := \conv f(\bar{\mathcal{X}})+K \subseteq \mathcal{P} \subseteq \conv f(\bar{\mathcal{X}})+K-\epsilon \{k\} =: \mathcal{P}_{\text{out}}.$$
	Note that by Proposition~\ref{prop:scalarization}, there exists $w_{x} \in K^+\setminus\{0\}$ such that $x \in \bar{\mathcal{X}}$ is an optimal solution to the weighted sum scalarization problem~\eqref{P_1(w)} for $w = w_x$. The algorithms in~\cite{cvop} returns also the set of these weight vectors $W = \{w_x \in K\setminus\{0\} \st x \in \bar{\mathcal{X}}\}$. Note that whenever a problem is not known to be bounded, the algorithms in~\cite{cvop} may be employed and as long as they return a solution, it is guaranteed that the problem is bounded and the solutions returned by the algorithm is correct. 
\end{remark}

If no ordering cone is given in~\eqref{(P)}, then it is taken as the positive orthant, that is, $K = \R^q_+$.

\subsection{Classical Utility Indifference Pricing}
\label{subsect:PrelimIndiff}
Utility indifference pricing under a complete preference, which is represented by the expectation of a utility function $u:\R \rightarrow \R \cup \{-\infty\}$ is well-defined and studied in the literature, see the overview by Henderson and Hobson~\cite{henderson} and references therein. Let $(\Omega,\mathcal{F},\P)$ be a probability space and $L^0(\mathcal{F},\R)$ be the set of all $\mathcal{F}$-measurable real-valued random variables. Recall that the utility indifference buy price $p^b \in \R$ is the price at which the investor is indifferent between paying nothing and not having claim $C_T \in L^0(\mathcal{F},\R)$, and paying $p^b$ at time $t=0$ to receive the claim at time $t=T$. In other words, $p^b$ is a solution of 
\begin{align*}
\sup_{V_T \in \mathcal{A}(x_0-p^b)}\E u(V_T+C_T) = \sup_{V_T \in \mathcal{A}(x_0)} \E u(V_T),
\end{align*}
where $x_0$ is the initial endowment, and $\mathcal{A}(\cdot)$ is the set of all wealth which can be generated from the corresponding initial wealth. Similarly, the utility indifference sell price $p^s \in \R$ is defined\footnote{There are alternative approaches to define the indifference buy and sell prices in the literature. Indeed, there is a recent discussion stating that the indifference prices provided above satisfy the so called ``complementary symmetry property", see for instance~\cite{Lew,chudziak2020}; and there are experiments showing that this property is systematically violated~(\cite{Birn}). Accordingly, it is possible to define, for instance, the utility indifference sell price as a solution of 
	\begin{align}
	\label{alterndef}
	\sup_{V_T \in \mathcal{A}(x_0)}\E u(V_T+C_T) = \sup_{V_T \in \mathcal{A}(x_0+p^s)} \E u(V_T),
	\end{align}
which accounts for the situation that one owns $C_T$ in order to sell it. Thus, the agent's initial pre-trade position is $(x_0, C_T)$, that is, $x_0$ at time zero, and $C_T$ initial wealth at time $T$. This would lead to an alternative description for $P^s$.   The definition in \eqref{eq:PbPs_scalar} corresponds to the situation, where the agent's initial pre-trade position is $(x_0, 0)$, that is, $x_0$ at time zero, and zero initial wealth at time T, see also \cite{henderson}. This could also be interpreted as leading to the indifference short-selling price, with \eqref{alterndef} as the indifference sell price. However, when we discuss the extensions of these concepts in Section~\ref{sect:indifference}, we keep the usual terminology and the sets as in \eqref{eq:PbPs_scalar}, since they are quite standard in Financial Mathematics, see for instance \cite{carmonabook,henderson,cheridito,benedetti2012,henderson2016}.}
as a solution of 
\begin{align*}
\sup_{V_T \in \mathcal{A}(x_0+p^s)}\E u(V_T-C_T) = \sup_{V_T \in \mathcal{A}(x_0)} \E u(V_T).
\end{align*}
Note that indifference buy and sell prices can be seen as the bounds on (buy and sell) prices for which one has a strict preference of buying and selling, respectively. Then, one can describe the utility indifference buy price as the {boundary} of the set $P^b$ of all prices at which buying the claim is at least as preferable as taking no action. Similarly, the utility indifference sell price is the {boundary} of the set $P^s$ of all prices at which selling the claim is at least as preferable as taking no action, respectively. More precisely, if we define
\begin{align}\label{eq:PbPs_scalar}
\begin{split}
P^b &:= \{p\in \R \: \st \sup_{V_T \in \mathcal{A}(x_0-p)}\E u(V_T+C_T) \geq \sup_{V_T \in \mathcal{A}(x_0)} \E u(V_T)\},\\
P^s &:= \{p\in\R \: \st \sup_{V_T \in \mathcal{A}(x_0+p)}\E u(V_T-C_T) \geq \sup_{V_T \in \mathcal{A}(x_0)} \E u(V_T)\},
\end{split}
\end{align}  
then, as long as prices $p^b$ and $p^s$ exists we have $P^b = (-\infty, p^b]$, $P^s = [p^s, \infty)$. Hence, $p^b = \bd P^b$ and $p^s = \bd P^s$. This point of view will be helpful when defining indifference prices for incomplete preference relations.

\subsection{Utility Representations for Incomplete Preferences}
\label{subsect:prelim1}
Let $(\Omega,\mathcal{F},\P)$ be a {finite} probability space, and $L^0(\mathcal{F},\R^d)$ be the set of all $\mathcal{F}$-measurable random variables which take their values in $\R^d$. {Denote the set of all continuous extended real-valued functions on $\R^d$ by $\mathcal{C}(\R^d)$,} and the set of all probability measures on $\Omega$ by $\mathcal{M}_1(\Omega)$. 

Throughout this paper, we consider the preference relations on the set $L^0(\mathcal{F},\R^d)$. Moreover, we consider a utility representation given as follows.
\begin{definition}\label{defn:ut_repr}
	A preference relation $\succsim$ on $L^0(\mathcal{F},\R^d)$ is said to admit a \emph{multi-prior expected multi-utility representation} if there exists a non-empty subset $\mathcal{U}$ of ${\mathcal{C}(\R^d)}$ and a non-empty subset $\mathcal{Q}$ of $\mathcal{M}_1(\Omega)$ such that, for random variables $Y, Z$ in $L^0(\mathcal{F},\R^d)$, we have 
	$$Y \succsim Z \iff \:\: \forall u\in\mathcal{U}, \forall Q \in\mathcal{Q}: \:\:\: \E_Q u(Y) \geq \E_Q u(Z).$$ 
\end{definition}

This type of representations\footnote{Note that it is also possible to consider the slighly more general preference relation in~\cite{Nau}, where there is a set of probability measure and utility pairs, say, $\mathcal{UQ}$ and $$Y \succsim Z \iff \:\: \forall (u,Q)\in\mathcal{UQ}: \:\:\: \E_Q u(Y) \geq \E_Q u(Z).$$ In this case, we would assume that there exists finitely many pairs in $\mathcal{UQ}$ instead of what is stated in Assumption~\ref{assump:utility}\emph{a.} However, keeping the representation as in Definition~\ref{defn:ut_repr} will be useful in simplifying some expressions throughout.\label{footnote:QUpairs}} for incomplete preferences are studied for instance in~\cite{Nau,ok2012,galaabaatar2013}. As a special case we also consider preference relations which admit a \emph{multi-prior expected single-utility representation} ($\mathcal{U}$ is a singleton) and a \emph{single-prior expected multi-utility representation} ($\mathcal{Q}$ is a singleton) as defined in~\cite{ok2012}.

\begin{remark}
	As usual we use the following notation throughout:
	$$Y \sim Z \iff Y\succsim Z \text{~and~} Z\succsim Y.$$ 
\end{remark}

In \cite{ok2012}, the necessary and sufficient conditions (assumptions both on the preference relation and on the set of the acts) for a preference relation to admit either a multi-prior expected single-utility or a single-prior expected multi-utility representation, where the prize space can be any {compact} metric space, are shown. Moreover, in~\cite{galaabaatar2013}, the characterization of multi-prior expected multi-utility representation, where the price space is not allowed to be $\R^d$ but it is a finite set, is given. Throughout this paper, we consider the multi-prior expected multi-utility representations of preference relations as given by Definition~\ref{defn:ut_repr}. Moreover, the functions $u \in \mathcal{U}$ are assumed to be multivariate utility functions defined  as follows\footnote{In~\cite{multivarutility}, Campi and Owen define a multivariate utility function in a similar way. Different from Definition~\ref{def:multivar_u}, they require $C_u:=\cl(\dom u)$ to be a convex cone such that $\R^d_+\subseteq C_u \neq \R^d$ and $u$ to be increasing with respect to the partial order $\leq_{C_u}$. Note that as $C_u\supseteq \R^d_+$, our definition is more general.}.

\begin{definition}\label{def:multivar_u}
	A proper concave function {$u: \R^d \rightarrow \R \cup \{\pm \infty\}$} is a multivariate utility function if $u$ is increasing with respect to the partial order $\leq$ on $\R^d$.
\end{definition}

{The preimage of the function $u$ is denoted by $u^{-1}$, that is, for $S \subseteq \R$ we have $u^{-1}(S) = \{x \in \R^d \st f(x) \in S\}.$ If $d=1$ and $u$ is invertible, then $u^{-1}(\cdot)$ corresponds to the inverse function as usual.}

\begin{assumption}\label{assump:utility}
	Throughout, we assume the following.
	\begin{enumerate}[a.]
		\item The preference relation admits a multi-prior expected multi-utility representation where $\mathcal{U}=\{u^1,\ldots, u^{r}\}$ and $\mathcal{Q} = \{Q^1 \ldots Q^{s}\}$ for some $r,s \geq 1$ with $q := rs$. 
		\item ${\mathcal{U}\subseteq \mathcal{C}(\R^d)}$ and any $u \in \mathcal{U}$ is a multivariate utility function.
		\item Any $u \in \mathcal{U}$ is strictly increasing in the sense that $x < y$ implies $u(x) < u(y)$. 
	\end{enumerate}
\end{assumption}

\section{Certainty Equivalent for Incomplete Preferences}
\label{sect:CE}

In the classical utility theory, where the preference relation is complete and represented by a single univariate utility function, the certainty equivalent of a random variable $Z$ is defined as the deterministic amount which would yield the same utility as the expected utility of $Z$. This amount is unique and can be computed if the utility function is bijective.

Under incomplete preferences, there is not necessarily a unique certainty equivalent of a random variable. In the past literature, usually a candidate with nice properties is picked and considered as the certainty equivalent. One of the choices is the \emph{worst-case (strong) certainty equivalent} when $d=1$. If $\mathcal{Q}$ is a singleton, i.e., the utility representation is given as a single-prior expected multi-utility representation, where the utility functions are bijective, then \emph{the strong certainty equivalent of $Z$} is given by $\inf_{u \in \mathcal{U}} u^{-1}(\E u(Z))$, see \cite{CEworst}. Similarly, one could consider the \emph{weak certainty equivalent}, namely, $\sup_{u \in \mathcal{U}} u^{-1}(\E u(Z))$. Applying the same idea to an incomplete preference that admits a multi-prior expected multi-utility representation for $d=1$, it is possible to consider the strong and the weak certainty equivalents given by $\inf_{Q \in \mathcal{Q},u \in \mathcal{U}} u^{-1}(\E_Q u(Z))$ and $\sup_{Q \in \mathcal{Q},u \in \mathcal{U}} u^{-1}(\E_Q u(Z))$, respectively. However, it is not clear if (or how) these strong and weak certainty equivalent concepts generalize to the case where $d > 1$ since the preimage $u^{-1}$ of a multivariate function $u$ yields a subset of $\R^d$ instead of a real number.

As already motivated in Section~\ref{sect:Intro} for a more general setting, we will now present the most intuitive definition of a certainty equivalent for the case $d \geq 1$, but we will see that this definition does not always provide a meaningful concept. Thus, instead, we will use the insights from Section~\ref{subsect:PrelimIndiff}, where we rewrote the scalar indifference prices as the upper and lower bounds of the set of all buy and sell prices, and we will see that this concept leads to a more suitable definition of a (weak and strong) certainty equivalent for the case $d \geq 1$.
 
We define the certainty equivalent of a random variable $Z \in L^0(\mathcal{F},\R^d)$ as a subset of $\R^d$, $d \geq 1$ as follows.
\begin{definition}\label{defn:CE}
	The certainty equivalent for $Z \in L^0(\mathcal{F},\R^d)$ is the set $$\C(Z) := \{c\in \R^d \st c \sim Z\}.$$
\end{definition} 
Let us consider the following sets
\begin{equation*} \label{eq:upperlowerCE}
\Cupp(Z) := \{c\in \R^d \st c \succsim Z\}, \:\:\:\:\:\: \Clow(Z):=\{c\in \R^d \st Z \succsim c\}.
\end{equation*}
Clearly, we have $\C(Z) = \Cupp(Z) \cap \Clow(Z)$. Note that it is highly possible that this intersection is empty as the preference relation is incomplete. In that case, considering $\Cupp(Z)$ and $\Clow(Z)$ would provide the full information for the decision maker.  

When we consider preferences which admit a multi-prior expected multi-utility representation, under Assumption~\ref{assump:utility}, these sets can be written as follows 
\begin{align}
\begin{split} \label{eq:CupClowCeq}
\Cupp(Z) 
&= \bigcap_{u \in \mathcal{U}}\{c \in \R^d \st u(c) \geq \sup_{Q\in\mathcal{Q}}\E_Q u(Z)\}; \\
\Clow(Z) 
&= \bigcap_{u \in \mathcal{U}}\{c \in \R^d \st u(c) \leq \inf_{Q\in\mathcal{Q}}\E_Q u(Z)\};\\
\C(Z) 
&= \bigcap_{u \in \mathcal{U},Q\in\mathcal{Q}}u^{-1}(\E_Q u(Z)).
\end{split}
\end{align}  

\begin{remark}\label{rem:CupClow}
	By continuity of the utility functions $u \in \mathcal{U}$, the sets $\Cupp(Z)$ and $\Clow(Z)$ are closed; by monotonicity of $u \in \mathcal{U}$, $\Cupp(Z)$ is an upper set and $\Clow(Z)$ is a lower set. Moreover, as $u \in \mathcal{U}$ are concave, $\Cupp(Z)$ is a convex set, whereas $\Clow(Z)$ is not convex in general. 
\end{remark}

\begin{proposition}\label{prop:CEnointerior}
	Under Assumption~\ref{assump:utility}, $\Int \C(Z) = \emptyset$ for any  $Z \in L^0(\mathcal{F},\R^d)$.	
\end{proposition}
\begin{proof}
	Assume the contrary and let $c \in \Int \C(Z)$. Then, there exists $\delta > 0 $ such that $c + \delta e \in \C(Z)$, where $e$ denotes the vector of ones. By Assumption~\ref{assump:utility} c. and by the definition of $\Cupp(Z)$, for all $u\in\mathcal{U}$ and for all $Q\in\mathcal{Q}$, we have $$u(c+\delta e) > u(c) \geq \E_Qu(Z).$$
	Hence, for all $u \in \mathcal{U}$, it is true that $u(c+\delta e) > \inf_{Q\in\mathcal{Q}}\E_Qu(Z)$. This implies that $c+\delta e \notin \Clow(Z)$, which is a contradiction to $c+\delta e \in \C(Z)$. 
\end{proof}

Note that in many cases $\C(Z)$ is an empty set, see e.g. Example~\ref{ex:1}, and thus not a suitable concept in general. Thus, we will propose an alternative definition that is based on the insights from Section~\ref{subsect:PrelimIndiff} and define the  strong and weak certainty equivalents as follows.
\begin{definition}\label{defn:strongweakCE}
	For $Z \in L^0(\mathcal{F},\R^d)$, the \emph{strong certainty equivalent} of $Z$ is $\Cs(Z):=\bd \Clow(Z)$ and the \emph{weak certainty equivalent} of $Z$ is $\Cw(Z):= \bd \Cupp(Z)$. 
\end{definition} 
The following proposition shows the characterizations and interpretations of the strong and weak certainty equivalents. The proof is quite standard and therefore omitted.
\begin{proposition} \label{prop:weakCE}
	Let $c \in \R^d$. Then, 
	\begin{enumerate} 
		\item $c \in \Cw(Z)$ if and only if
		\begin{enumerate}[i.]
			\item $c\succsim Z$ and
			\item $c-\eps\not\succsim Z$ for all $\eps \in \Int \R^d_+$;
		\end{enumerate}
		\item $c \in \Cs(Z)$ if and only if
		\begin{enumerate}[i.]
			\item $Z\succsim c$ and 
			\item $Z\not\succsim c+\eps$ for all $\eps \in \Int \R^d_+$. 
		\end{enumerate}
	\end{enumerate}
\end{proposition}


In the following two remarks we consider the two special cases $d = 1$ and $\mathcal{U} = \{u\}$.
\begin{remark} \label{rem:CE_d=1} 
	If the price space is $\R$, we have $\Cupp(Z) = [c^w,\infty)$ and $\Clow = (-\infty,c^s]$ for some $c^w,c^s \in \R$, see Remark~\ref{rem:CupClow}. By Proposition~\ref{prop:CEnointerior}, we have $c^s \leq c^w$. Moreover, since the utility functions $u\in\mathcal U$ are strictly increasing, the inverse function $u^{-1}$ is well defined. Indeed, by monotonicity of $u$, we have
	\begin{align*}
	\Cs(Z) = \inf_{u\in\mathcal{U},Q\in\mathcal{Q}}u^{-1}(\E_Qu(Z)) \text{~~and~~} \Cw(Z) = \sup_{u\in\mathcal{U},Q\in\mathcal{Q}}u^{-1}(\E_Qu(Z)).
	\end{align*}
	Hence, we recover the strong and the weak certainty equivalents as mentioned in the beginning of Section~\ref{sect:CE}. 
	
	When restricted to $\mathcal{Q}$ being a singleton, this definition yields the strong certainty equivalent introduced in~\cite{CEworst}. 
	
	Moreover, $\C(Z)\neq \emptyset$, if and only if $c:=c^w=c^s = u^{-1}(\E_Qu(Z))$ for all $u\in\mathcal{U}$, $Q\in\mathcal{Q}$. In this case, we have $C(Z) = \{c\}$. This observation also proves the recovery of the classical certainty equivalent whenever a complete preference admitting a von Neumann-Morgenstern utility (single-prior expected single-univariate-utility) representation is considered.   
\end{remark}
\begin{remark} \label{rem:CE_mpsu}
	If the preference relation admits a multi-prior expected single-utility representation, that is $\mathcal{U} = \{u\}$, then for any $Z \in L^0(\mathcal{F}, \R^d)$ we have
	\begin{align*}
	\Cupp(Z) = u^{-1}\bigg( \big[\sup_{Q\in \mathcal{Q}} \E_Q u(Z), \infty \big) \bigg) \text{~~and~~}
	\Clow(Z) =  u^{-1}\bigg( \big(-\infty, \inf_{Q\in \mathcal{Q}} \E_Q u(Z) \big] \bigg),
	\end{align*}  
where $u^{-1}$ is the preimage. Moreover, by the monotonicity and continuity of $u$ 
	\begin{equation*} \label{eq:C_mpsu}
	\Cw(Z) = u^{-1}\big(\sup_{Q\in \mathcal{Q}} \E_Q u(Z)\big), \:\: \Cs(Z) = u^{-1} \big(\inf_{Q\in \mathcal{Q}} \E_Q u(Z) \big). 
	\end{equation*}
	
	Note that if the preference relation is complete and admits a single-prior expected single-utility representation, that is, if $\mathcal{Q} = \{Q\}$ and $\mathcal{U} = \{u\}$, then we have 
	\begin{equation*} \label{eq:singleQ}
	\C(Z) = \Cw(Z) = \Cs(Z) = u^{-1} (\E_Q u(Z)).
	\end{equation*}  
	This suggests that for a complete preference relation represented by a single multivariate utility function $u:\R^d \rightarrow \R \cup \{-\infty\}$, the certainty equivalent of $Z$ is defined as the preimage $u^{-1}(\E_Q u(Z)) \subseteq \R^d$. 
\end{remark}

\section{Utility Indifference Pricing for Incomplete Preferences}
\label{sect:indifference}

In this section, we consider the indifference pricing problem where the preference relation is not necessarily complete. In particular, we consider the case where Assumption~\ref{assump:utility} holds. Following the footsteps of the classical definition, we first consider the `utility maximization problem' for such representations of the incomplete preferences. 

\begin{notation}
	\label{notation:1}
	We denote the vector-valued expected utility functional by $U(\cdot):L^0(\mathcal{F},\R^d) \rightarrow \R^q$, where $U(\cdot):= (\E_{Q_1} u^1(\cdot), \ldots,\E_{Q_1} u^{r}(\cdot),\ldots, \E_{Q_{s}} u^1(\cdot), \ldots \E_{Q_{s}} u^{r}(\cdot))^T$. 
\end{notation}

Now, under Assumption~\ref{assump:utility}\footnote{If we consider a representation given by a set of probability measures paired with utility functions as in \cite{Nau}, we would list all the pairs in order to obtain $U(\cdot)$ and all the results of this section would remain the same, see also Footnote~\ref{footnote:QUpairs}.}, the utility maximization problem can be seen as a vector optimization problem $P(x,C_T)$ given by
\begin{align}\begin{split}\label{(U)}
P(x,C_T):~~~~~\text{maximize~~~~~} U(Z+C_T) \text{~~subject to~~~}  Z \in \mathcal{A}(x),
\end{split}
\end{align}
where $\mathcal{A}(x)\subseteq L^0(\mathcal{F}_T,\R^d)$ is the set of all wealth that can be generated from initial endowment $x$, and $C_T \in L^0(\mathcal{F}_T,\R^d)$ is some payoff that is received at time $T$. Note that the ordering cone for this problem is the positive orthant. This is because an alternative with component-wise larger expected utility would be preferred by the decision maker.

Throughout, we assume that $\mathcal{A}(\cdot)$ satisfies the following.
\begin{assumption} \label{assmp:A} Let $x, y \in \R^d$, $\lambda \in [0,1]$ be arbitrary.
	\begin{enumerate}[a.]
		\item $\mathcal{A}(x)$ is a convex set.
		\item $\lambda \mathcal{A}(x) + (1-\lambda)\mathcal{A}(y) \subseteq \mathcal{A}(\lambda x + (1-\lambda)y)$, where the set addition and multiplication are the usual Minkowski operations.
		\item If $x \leq y$, then $\mathcal{A}(x) \subseteq \mathcal{A}(y)$. 
		\item If $V_T \in \mathcal{A}(x)$, then $V_T+r \in \mathcal{A}(x+r)$ for any $r \in \R^d$.
		\item {Let $(x^k)_{k\geq 1} \in \R^d$ be a decreasing sequence with respect to $\leq$ with $\lim_{k\rightarrow \infty} x^k = x \in \R^d$. Then, $\mathcal{A}(x)= \bigcap_{k\geq 1}\mathcal{A}(x^k)$.}
	\end{enumerate}
\end{assumption}

Two different market models and thus examples for $\mathcal{A}(x)$ will be given in Sections~\ref{sect:univariatecase} and~\ref{sect:Conical}. Note that by Assumption~\ref{assump:utility}, $u\in\mathcal{U}$ are concave, and by Assumption~\ref{assmp:A} a., $\mathcal{A}(x)$ is a convex set. Then,~\eqref{(U)} is the negative of the following convex vector optimization problem
\begin{equation}\label{eq:Uconvex}
\text{minimize~} -U(Z+C_T) \text{~~subject to~}  Z \in \mathcal{A}(x), \end{equation}
and the lower image of~\eqref{(U)} is equal to the negative of the upper image of the convex vector optimization problem given by~\eqref{eq:Uconvex}.

As introduced in Section~\ref{subsect:prelimCVOP}, there is no single optimal objective value of~\eqref{(U)} and we consider the set of all (weakly) maximal elements of the lower image. The lower image of problem~\eqref{(U)} can be written as the following set-valued function
\begin{equation}\label{eq:valuefunction}
V(x,C_T):=\cl \bigcup_{V_T \in \mathcal{A}(x)} \left(U(V_T+C_T) - \R^q_+\right).
\end{equation}

\begin{remark}\label{prop:V_properties}
	Let $C_T, \tilde{C}_T \in L(\mathcal{F}_T,\R^d)$ and $x, y \in \R^d$. Then, the following implications hold.  
	\begin{enumerate}[a.]
		\item Assumption~\ref{assump:utility} b. implies that if $C_T \leq \tilde{C}_T$, then $V(x,C_T) \subseteq V(x,\tilde{C}_T)$;
		\item Assumption~\ref{assump:utility} c. implies that if $C_T < \tilde{C}_T$, then $V(x,C_T) \subsetneq V(x,\tilde{C}_T)$;
		\item Assumption~\ref{assmp:A} c. implies that if $x \leq y$, then $V(x,C_T) \subseteq V(y, C_T)$;
		\item Assumption~\ref{assump:utility} c. and Assumption~\ref{assmp:A} d.  imply that if $x < y$, then $V(x,C_T) \subsetneq V(y, C_T)$. To see that, note that Assumption~\ref{assmp:A} d. implies that if $x < y$, then for all $V_T\in\mathcal{A}(x)$ there exists $\tilde{V}_T\in\mathcal{A}(y)$ such that $V_T < \tilde{V}_T$.
	\end{enumerate}
\end{remark}

As in the usual utility indifference pricing theory, we first consider the problems $V(x_0 - p^b, C_T)$, $V(x_0 + p^s, -C_T)$ and $V(x_0,0)$, where $p^b, p^s \in \R^d$ are candidates of indifference buy and sell prices of the claim $C_T$, respectively. Different from the scalar case, the existence of $p^b$ and $p^s$ that would satisfy $V(x_0 - p^b, C_T) = V(x_0, 0)$, respectively $V(x_0+p^s, -C_T) = V(x_0, 0)$, is not guaranteed. Thus, instead, we will base our definition of the set-valued buy and sell prices on the reformulation of the scalar indifference price given by~\eqref{eq:PbPs_scalar}. In other words, we consider the set of all prices at which one would prefer buying the claim compared to taking no action. Similarly, we consider the set of all prices at which selling the claim is preferred compared to taking no action.  Then, the indifference prices will be defined as the boundaries of those sets.

We suggest that buying the claim $C_T$ at price $p \in \R^d$ is at least as preferred as not buying it if 
\begin{equation} \label{eq:buypreference}
\Max V(x_0,0) \curlyeqprec \Max V(x_0-p,C_T)
\end{equation} holds. Indeed, as the lower images $V(\cdot,\cdot)$ are closed lower sets,~\eqref{eq:buypreference} holds if and only if 
$V(x_0,0)\curlyeqprec V(x_0-p,C_T)$, or equivalently,
\begin{equation}\label{eq:strictbuy}
V(x_0 - p, C_T) \supseteq V(x_0, 0)
\end{equation}  
holds. Similarly, selling $C_T$ at price $p\in\R^d$ is preferred to taking no action if 
\begin{equation} \label{eq:strictsell}
V(x_0 + p, -C_T) \supseteq V(x_0, 0).
\end{equation}

\begin{remark}\label{rem:characterization}
	By Proposition~\ref{prop:scalarization},~\eqref{eq:strictbuy} implies that 
	\begin{equation} \label{eq:scalarcharacterizationbuy}
	\sup_{V_T\in \mathcal{A}(x_0-p)}w^TU(V_T+C_T) \geq \sup_{V_T \in \mathcal{A}(x_0)}w^TU(V_T)
	\end{equation}
	holds for all $w \in \R^d_+$. Moreover, the reverse implication holds if $\mathcal{A}$ is a closed set. In this case, satisfying \eqref{eq:scalarcharacterizationbuy} for all $w \in\R^d_+$ can be seen as the characterization of~\eqref{eq:strictbuy}. A similar characterization can be written for~\eqref{eq:strictsell}. \end{remark}

We define the set-valued buy and sell prices as follows \footnote{Following the remark given in Footnote 1, an alternative definition for the set-valued sell price of $C_T$ would be $\tilde{P}^s(C_T) =\{p^s \in \R^d \st V(x_0 + p^s, 0) \supseteq V(x_0, C_T)\}$. With this definition, Remark~\ref{rem:Pp=-Ps} is not correct anymore. Hence, one needs to check the rest of the results in Section~\ref{sect:indifference} separately for $\tilde{P}^s(C_T)$. It is straightforward to see that Propositions~\ref{prop:lowerconvexset}-1., \ref{prop:lowerconvexset}-2. and \ref{prop:closedness} hold correct for this definition. Moreover, both the statements and the proofs of Propositions~\ref{prop:priceboundBuy} and \ref{prop:uniformcont} can be modified accordingly. However, the steps followed to prove Propositions~\ref{prop:lowerconvexset}-{3.} and \ref{prop:intersectionPbPs} can not be applied directly to the alternative definition. Note that as Proposition~\ref{prop:lowerconvexset}-{1.} holds correct, the computations of $\tilde{P}^s(C_T)$ can be done by applying similar techniques as described in Section~\ref{sect:computations}.}.

\begin{definition} \label{defn:strictprices}
	The set-valued buy price of $C_T$, $P^b(C_T)$, is the set of all prices $p^b\in \R^d$ satisfying~\eqref{eq:strictbuy}, and the set-valued sell price of $C_T$, $P^s(C_T)$, is the set of all prices $p^s\in \R^d$ satisfying~\eqref{eq:strictsell}. That is,
	\begin{align*}
	P^b(C_T)&: =\{p^b \in \R^d \st V(x_0 - p^b, C_T) \supseteq V(x_0, 0)\},\\
	P^s(C_T)&: =\{p^s \in \R^d \st V(x_0 + p^s, -C_T) \supseteq V(x_0, 0)\}.
	\end{align*} 
\end{definition}
\begin{remark}\label{rem:Pp=-Ps}
	By the definition, it is true that $P^b(C_T)= -P^s(-C_T)$. Hence, in Propositions~\ref{prop:lowerconvexset} and \ref{prop:closedness}, the statements are proven for the set-valued buy price $P^b(\cdot)$ only.  
\end{remark}

Note that the set-valued buy/sell prices defined above are not \emph{indifference} buy/sell prices. Indeed, for any element $p$ of $P^b(C_T)$/$P^s(C_T)$, it is \emph{better} for the decision maker to buy/sell the claim at that price. Hence, one may even call these, \emph{set-valued better to buy/sell} prices in order to emphasize this observation. For simplicity, we keep the names as they are.

Below we will show that $P^b(\cdot)$ and $P^s(\cdot)$ satisfy some properties which are in parallel to the properties of the scalar buy and sell prices under complete preferences{\footnote{Note also the relationship to the definition of the certainty equivalent, in particular between $\Cupp(C_T)$ and $P^s(C_T)$. A certain amount $c\in \bd \Cupp(C_T)= \Cw(C_T)$ is preferred to $C_T$ ($c \succsim C_T$), but for any $\eps\in\Int\R^d_+$, $c-\eps$ is not anymore preferred to it ($c-\eps\not\succsim C_T$). Similarly, for a price $p\in\bd P^s(C_T)$, the decision maker would prefer selling the claim at that price rather than not taking any action, but for any $\eps\in\Int\R^d_+$, $p-\eps$ is not anymore a sell price for him/her (see Propositions~\ref{prop:weakCE}-1. and~\ref{prop:priceboundBuy}-2. below). Moreover both sets are convex upper sets. The situation is somehow different when we consider $\Clow(C_T)$. This set is not necessarily convex (unlike $P^b(C_T)$). Moreover, for any $c\in\Clow(C_T)$, $C_T$ is preferred to $c$ (not the other way around), and for any $\eps\in\Int\R^d_+$, $C_T$ is not anymore preferred to $c+\eps$.}. 

First, we show that $P^b(C_T), P^s(C_T) \subseteq \R^d$ are lower, respectively upper, convex sets for any $C_T \in L^0(\mathcal{F}_T,\R^d)$. Furthermore, we show the monotonicity of both price functions as well as the concavity of $P^b(\cdot)$ and the convexity of $P^s(\cdot)$ in the sense of set-valued functions. The proof can be found in the Appendix.

\begin{proposition}\label{prop:lowerconvexset}
	Let Assumptions~\ref{assump:utility} and~\ref{assmp:A} a-c hold. 
	\begin{enumerate}
		\item For a claim $C_T \in L(\mathcal{F}_T, \R^d)$, $P^b(C_T)$ is a convex lower set and $P^s(C_T)$ is a convex upper set.
		\item $P^b(\cdot)$ and $P^s(\cdot)$ are increasing with respect to the partial order $\leq$, in the sense of set orders $\curlyeqprec$ and $\preccurlyeq$, respectively: For  $C_T^1, C_T^2 \in L(\mathcal{F}_T, \R^d)$, if $C_T^1 \leq C_T^2$, then $P^b(C_T^1) \curlyeqprec P^b(C_T^2)$ and $P^s(C_T^1) \preccurlyeq P^s(C_T^2)$.
		\item  $P^b(\cdot)$ is concave with respect to $\curlyeqprec$: For $C_T^1, C_T^2 \in L(\mathcal{F}_T, \R^d)$ and $\lambda \in [0,1]$
		\begin{equation} \label{eq:Pbconcave}
		\lambda P^b(C_T^1)+(1-\lambda)P^b(C_T^2) \curlyeqprec P^b(C_T^{\lambda})
		\end{equation}
		holds, where $C_T^{\lambda}:= \lambda C_T^1 + (1-\lambda)C_T^2$. Similarly, $P^s(\cdot)$ is convex with respect to $\preccurlyeq$.
	\end{enumerate}
\end{proposition}

The properties proven in Proposition~\ref{prop:lowerconvexset} simplify further whenever $d = 1$. First, note that $P^b(C_T)$ and $P^s(C_T)$ are then intervals by Proposition~\ref{prop:lowerconvexset}. Moreover, if the preference relation is complete and a von Neumann and Morgenstern utility representation is given by $u:\R\rightarrow \R \cup \{-\infty\}$, then one recovers the usual definition and the properties of the indifference prices. Indeed, $P^b(C_T), P^s(C_T)$ simplify to $P^b, P^s$ given by~\eqref{eq:PbPs_scalar}. Then, $\sup P^b(C_T)=\bd P^b(C_T)$ is the classical utility indifference buy price and {$\inf P^s(C_T) = \bd P^s(C_T)$} is the classical utility indifference sell price. In this case, assertions 2. and~3. of Proposition~\ref{prop:lowerconvexset} simply recover the monotonicity and concavity (convexity) of the utility indifference buy (sell) price.

By the following propositions, proofs of which can be found in the Appendix, we show that under some additional assumptions on the market model, namely Assumptions~\ref{assmp:A} d. to e., buy and sell prices are closed sets and the intersection of buy and sell prices has no interior. Then, we define indifference price bounds as the boundaries of the set-valued buy and sell prices, namely $\bd P^b(C_T)$ and $\bd P^s(C_T)$.

\begin{proposition}\label{prop:intersectionPbPs}
	Let Assumptions~\ref{assump:utility} and~\ref{assmp:A} a-d hold. Then, {for any $C_T\in L^0(\mathcal{F}_T,\R^d)$, the followings hold
	\begin{enumerate}
		\item If $p\in P^b(C_T)\cap P^s(C_T)$, then $V(x_0-p,C_T) = V(x_0,0) = V(x_0+p,-C_T)$; 
		\item $\Int (P^b(C_T)\cap P^s(C_T)) = \emptyset$.
	\end{enumerate}}
\end{proposition}

\begin{proposition} \label{prop:closedness}
	Let Assumptions~\ref{assump:utility} and~\ref{assmp:A} hold. For a claim $C_T \in L(\mathcal{F}_T, \R^d)$, the set-valued buy and sell prices $P^b(C_T)$ and $P^s(C_T)$ are closed subsets of $\R^d$.
\end{proposition}

Now as set-valued buy and sell prices are closed convex lower, respectively upper sets that do not have a solid intersection, we define indifference price bounds as the boundaries of these set-valued prices. 

\begin{definition}\label{defn:indifpricebounds}  
	\emph{The indifference price bounds} for $C_T$ are $\bd P^b(C_T)$ and $\bd P^s(C_T)$.
\end{definition}

Note that the definition of the indifference price bounds are similar to the definitions of the strong and weak certainty equivalents in a way that they are boundaries of lower and upper closed sets, respectively. The following proposition, similar to Proposition~\ref{prop:weakCE} for strong and weak certainty equivalents, shows the motivation behind the definition for the indifference price bounds. 
\begin{proposition} \label{prop:priceboundBuy}
	Let Assumptions~\ref{assmp:A} and~\ref{assump:utility} hold. Let $p \in \R^d$. Then, 
	\begin{enumerate}
		\item $p \in \bd P^b(C_T)$ if and only if the followings hold:
		\begin{enumerate}[i.]
			\item $V(x_0-p,C_T) \supseteq V(x_0,0)$;
			\item For any $\eps \in \Int \R^d_+$ it is true that $V(x_0-p-\eps,C_T) \nsupseteq V(x_0,0)$; 
		\end{enumerate}
		\item $p \in \bd P^s(C_T)$ if and only if the followings hold:
		\begin{enumerate}[i.]
			\item $V(x_0+p,-C_T) \supseteq V(x_0,0)$;
			\item For any $\eps \in \Int \R^d_+$ it is true that $V(x_0+p-\eps,-C_T) \nsupseteq V(x_0,0)$. 
		\end{enumerate}
	\end{enumerate}
\end{proposition}
\begin{proof}
	By Propositions~\ref{prop:lowerconvexset} and~\ref{prop:closedness} we know that $P^b(C_T)$ is a lower closed set and $P^s(C_T)$ is an upper closed set. Hence, $\wMax P^b(C_T) = \bd P^b(C_T)$ and $ \wMin P^s(C_T) = \bd P^s(C_T)$. The the assertion follows from the definitions of weakly maximal and weakly minimal elements. 
\end{proof}

Note that for any $p \in \bd P^b(C_T)$ it holds $V(x_0-p,C_T) \supseteq V(x_0,0)$, that is, buying the claim at $p$ is at least as preferred as not buying it. Moreover, by Remark~\ref{rem:characterization}, if $\mathcal{A}$ is closed and the utility maximization problem is bounded,  $V(x_0-p-\epsilon,C_T) \nsupseteq V(x_0,0)$ implies that there exists $w \in \R^q_+$ such that the maximum expected weighted utility is strictly less if one buys the claim at $p+\epsilon$, that is, $$\sup_{V_T \in \mathcal{A}(x_0-p-\epsilon)}w^TU(V_T+C_T) < \sup_{V_T \in \mathcal{A}(x_0)}w^TU(V_T).$$
Similarly, for any $p \in \bd P^s(C_T)$, selling the claim at $p$ is at least as preferred as not selling it. However, for any $\epsilon \in \Int \R^q_+$, there exists $w \in \R^q_+$ such that the maximum expected weighted utility is strictly less if one sells the claim at $p-\epsilon$, that is, $$\sup_{V_T \in \mathcal{A}(x_0+p-\epsilon)}w^TU(V_T-C_T) < \sup_{V_T \in \mathcal{A}(x_0)}w^TU(V_T).$$

{With the next proposition, we show that under some further assumptions on $u\in\mathcal{U}$ and $\mathcal{A}(\cdot)$, for any $p \in \bd P^b(C_T)$, there exists a weight vector $w\in\R^q_+$ such that paying $p$ to receive $C_T$ and paying nothing and not having $C_T$ have the same maximum expected weighted utility $w^TU$. The proof can be found in the Appendix.} 
{\begin{proposition}\label{prop:uniformcont}
		If each $u\in\mathcal{U}$ is uniformly continuous and $\mathcal{A}(x)=x+\mathcal{A}(0)$ for all $x\in \R^d$, then $V(x_0,0)\nsubseteq \Int V(x_0-p,C_T)$ for any $p\in\bd P^b(C_T)$. Similarly, $V(x_0,0)\nsubseteq \Int V(x_0+p,-C_T)$ for any $p\in\bd P^s(C_T)$. 
\end{proposition}} 
{Proposition~\ref{prop:uniformcont} shows that the boundaries of $V(x_0,0)$ and $V(x_0-p,C_T)$ intersect, hence for $p\in\bd P^b(C_T)$ there exists $w \in\R^q_+$ such that 
	\begin{equation}\label{eq:weigtedUbuy}
	\sup_{V_T\in\mathcal{A}(x_0-p)}w^TU(V_T+C_T) = \sup_{V_T\in\mathcal{A}(x_0)}w^TU(V_T).
	\end{equation} 
	Similarly, if $p\in\bd P^s(C_T)$, then there exists $w\in\R^q_+$ such that
	\begin{equation*}\label{eq:weigtedUsell}
	\sup_{V_T\in\mathcal{A}(x_0+p)}w^TU(V_T-C_T) = \sup_{V_T\in\mathcal{A}(x_0)}w^TU(V_T).
	\end{equation*}}

{Note that the market models explained in Section~\ref{sect:univariatecase} and Section~\ref{sect:Conical} satisfy the assumption $\mathcal{A}(x)=x+\mathcal{A}(0)$ for all $x\in \R^d$. Moreover, the utility functions that are considered in Example~\ref{ex:comp2dim} are uniformly continuous.}

{\begin{remark}\label{rem:Defn1}
		In Definition~\ref{defn:indifpricebounds}, the boundaries of the sets $P^b(C_T)$ and $P^s(C_T)$ are called the indifference price bounds for $C_T$. Note that different from the scalar case, for example, the buyer of claim $C_T$ is not really \emph{indifferent} between `paying nothing and not having $C_T$' and `paying $p^b \in \bd P^b(C_T)$ to receive $C_T$'. However, in the special case of a complete preference relation with $d = 1$, these sets reduce to the usual indifference prices. Moreover, when restricted to the special case of a complete preference relation with $d>1$ under the conical market model, these sets contain the indifference prices as defined in~\cite{benedetti2012}, see Section~\ref{subsect:complete}.\\
		Furthermore, for the general case, by Proposition~\ref{prop:uniformcont} and \eqref{eq:weigtedUbuy}, we observe that if $p\in\bd P^b(C_T)$, then a decision maker with a complete preference relation which admits a particular weighted sum of the vector valued utility, $w^TU$, as its representation, would be indifferent between the two options.\\
		In economic terms, the sets $P^b(C_T)$ and $P^s(C_T)$ can be seen as the willingness to pay and then the boundaries would be the reservation price or sell/buy price, which is called the indifference price in Finance. Thus, we decided to still call $\bd P^b(C_T)$ and $\bd P^s(C_T)$ the indifference price bounds in analogy to the scalar case, knowing that it does in general not mean being indifferent as in the classical sense, but more in the sense of Proposition~\ref{prop:uniformcont}.
\end{remark}}

\begin{remark}\label{rem:SHP}
	In~\cite{SHP}, L\"ohne and Rudloff study the set of all superhedging portfolios for num\'eraire free markets with transactions costs and provide an algorithm to compute it. Accordingly, for a claim $C_T$, the set of superhedging portfolios is given by $$\SHP(C_T):=\{p\in \R^d \st C_T \in \mathcal{A}(p)\}$$ and the set of all subhedging portfolios for $C_T$ is $$\SubHP(C_T) := -\SHP(-C_T).$$ 
	Note that for $d=1$, these sets would be intervals leading to the usual no-arbitrage pricing interval given by $(\sup \SubHP(C_T), \inf \SHP(C_T)).$
	
	If $\mathcal{A}(0)+\mathcal{A}(0)\subseteq \mathcal{A}(0)$, which is the case for the conical market model also considered in~\cite{SHP}, we have 
	\begin{align}\label{superhedge}
	P^s(C_T) \supseteq \SHP(C_T) \text{~~and~~} P^b(C_T) \supseteq \SubHP(C_T).
	\end{align}
	Indeed, for $p\in \SHP(C_T)$, we have $C_T\in\mathcal{A}(p)$. By Assumption~\ref{assmp:A} d, and $\mathcal{A}(0)+\mathcal{A}(0)\subseteq \mathcal{A}(0)$, $V_T+C_T\in\mathcal{A}(x_0+p)$ for any $V_T\in\mathcal{A}(x_0)$. This implies $V(x_0,0)\subseteq V(p,-C_T)$. To see that, let $U(V_T)-r \in V(x_0,0)$ for some $V_T\in\mathcal{A}(x_0)$, $r\in\R^q_+$. Note that $U(V_T)-r=U(V_T+C_T-C_T)-r \in V(p,-C_T)$ as $V_T+C_T \in\mathcal{A}(x_0+p)$. The second inclusion can be shown symmetrically.  
	
	It is well known that in incomplete financial markets, superhedging can be quite expensive and thus the interval or set of no-arbitrage prices can be quite big. Indifference pricing leads then to smaller price intervals. Equation~\eqref{superhedge} confirms that this is also the case when incomplete preference relations are considered. In Examples~\ref{ex:1}, \ref{ex:multivarcomplete} and~\ref{ex:comp2dim}, the utility indifference price bounds and for comparison also the super- and subhedging price bounds will be computed to illustrate the relationship given in~\eqref{superhedge}.
\end{remark}

\section{Computing the Certainty Equivalent and Indifference Price Bounds}\label{sect:computations}
Before considering different market models and solving numerical examples in Section~\ref{sect:appl}, we now discuss the computations of the set-valued quantities introduced in Sections~\ref{sect:CE} and \ref{sect:indifference}. Note that the computations are related to solving CVOPs and {we will show some simplifications} for some special cases. First, we discuss computing the certainty equivalent and then approximations to the indifference price bounds. Numerical examples will be given in Section~\ref{sect:appl}, see Examples~\ref{ex:1}, ~\ref{ex:multivarcomplete}, and~\ref{ex:comp2dim}.

\subsection{Computing $\Cupp(Z)$ and $\Clow(Z)$}
\label{subsect:cCEcomputation}
The computations of $\Cupp(Z)$ and $\Clow(Z)$ for $d=1$ are already given by Remark~\ref{rem:CE_d=1}. Here, we focus on the $d>1$ case only. As stated in Remark~\ref{rem:CupClow}, $\Cupp(Z)$ is a closed upper set. Indeed, using the representation given in~\eqref{eq:CupClowCeq}, it is easy to see that $\Cupp(Z)$ is the upper image of the following convex vector optimization problem with $r$ constraints: 
\begin{align}\begin{split} \label{eq:Cup_cvop}
&\text{minimize~~~~~} c \\ 
&\text{subject to~~~} \sup_{Q\in\mathcal{Q}}\E_Qu(Z)-u(c) \leq 0 \text{~~~ for all~} u\in \mathcal{U}. 
\end{split}
\end{align}

On the other hand, even though it is known by Remark~\ref{rem:CupClow} that $\Clow(Z)$ is a closed lower set, computing $\Clow(Z)$ requires more computational effort than computing $\Cupp(Z)$, in general. One can show that $\Clow(Z)$ is the lower image of the following vector optimization problem
\begin{align}\begin{split} \label{eq:Clow_noncvop}
&\text{maximize~~~~~} c \\
&\text{subject to~~~} \inf_{Q\in\mathcal{Q}}\E_Qu(Z)-u(c) \geq 0 \text{~~~ for all~} u\in \mathcal{U}. 
\end{split}
\end{align}
This problem is non-convex if the utility functions are not linear. There are algorithms that approximately solve non-convex vector optimization problems, see~\cite{nonconvex}. Instead of solving one non-convex VOP, one can also solve $r$ convex vector optimization problems in order to generate $\Clow(Z)$. Note that by the continuity of $u\in\mathcal{U}$, we have
\begin{align*}
\cl(\R^d\setminus\Clow(Z)) &= \bigcup_{u\in\mathcal{U}}\{c\in\R^d \st u(c) \geq \inf_{Q\in\mathcal{Q}}\E_Q u(Z)\},
\end{align*}
and each set $\{c\in\R^d \st u(c) \geq \inf_{Q\in\mathcal{Q}}\E_Q u(Z)\}$ is the upper image of the following vector optimization problem 
\begin{align}\begin{split} \label{eq:Clow_cvop}
&\text{minimize~~~~~} c \\
&\text{subject to~~~} \inf_{Q\in\mathcal{Q}}\E_Q u(Z)-u(c) \leq 0.
\end{split}
\end{align}
Then, one needs to solve $r$ convex vector optimization problems (one for each $u \in\mathcal{U}$), and the union of the upper images over all $u\in\mathcal{U}$ yields $\cl(\R^d\setminus\Clow(Z))$. 

Note that if the preference relation admits a multi-prior expected single utility representation, that is, if $r = 1$, then clearly it is enough to solve a single CVOP to compute $\Clow(Z)$. 

\begin{remark}\label{rem:CEcompletemulvarU}
	If the preference relation admits a single-prior expected-single utility representation where $\mathcal{U}=\{u\},\mathcal{Q}=\{Q\}$ and $d >1$, see also Remark~\ref{rem:CE_mpsu}, then, $\Cupp(Z)=\cl(\R^d\setminus\Clow(Z))$ and $C(Z)$ is the boundary of the upper image of the following convex vector optimization problem
	\begin{align*}
	&\text{minimize~~~~~} c \\
	&\text{subject to~~~} \E u(Z)-u(c) \leq 0. 
	\end{align*}
\end{remark}

\subsection{Computations of the Buy and Sell Price Bounds}
\label{subsubsect:strictcompute}
It is known by Proposition~\ref{prop:lowerconvexset} that set-valued buy and sell prices are lower, respectively upper {closed} convex sets. The main idea is that these sets can be seen as lower, respectively upper images of a certain convex vector optimization problem. Then, the aim is to solve these CVOPs' in order to find inner and outer approximations to the set-valued buy and sell prices. 

The first step is to solve the utility maximization problem~\eqref{(U)} for $C_T=0$ and $x=x_0$ using a CVOP algorithm to obtain an inner and an outer approximation to the lower image $V(x_0,0)$ of problem~\eqref{(U)} as defined in~\eqref{eq:valuefunction}. Note that for bounded problems, the primal as well as the dual algorithm provided in~\cite{cvop} yields a finite weak $\epsilon$-solution $\bar{\mathcal{X}}= \{{X}^1,\ldots,{X}^l\}\subseteq \mathcal{A}(x_0)$ of~$P(x_0,0)$ defined in \eqref{(U)}  in the sense of Definition~\ref{defn:solnCVOP}. Hence, it is true that \begin{equation}\label{eq:innerouterappofupper}
\conv U(\bar{\mathcal{X}}) - \R^q_+ \subseteq V(x_0,0) \subseteq \conv U(\bar{\mathcal{X}}) - \R^q_+ + \epsilon k,
\end{equation}
where $\epsilon > 0$ is the approximation error bound and $k \in \Int\R^q_+$ is fixed.

Moreover, by the structure of these algorithms, ${X}^i \in \bar{\mathcal{X}}$ is an optimal solution of the weighted sum scalarization problem for some ${w}^i \in \R^d_+$, that is, $$({w}^i)^TU({X}^i) = \max_{V_T\in \mathcal{A}(x_0)} ({w}^i)^TU(V_T) =: v^{w^i}.$$ The algorithms in~\cite{cvop} also provide these weight vectors ${w}^i \in \R^d_+$ for ${X}^i \in \bar{\mathcal{X}}$, see also Remark~\ref{rem:innerouterappr}. Let the finite set of weight vectors provided by the algorithm be ${W} :=\{{w}^1,\ldots,{w}^l\}$. 
In the following two sections, we provide methods to compute a superset and a subset of $P^b(C_T)$ and $P^s(C_T)$ using such  a finite weak $\epsilon$-solution $\bar{\mathcal{X}}$ as well as the finite set of weight vectors ${W}$.

\subsubsection{Computing a Superset of $P^b(C_T)$ and $P^s(C_T)$} \label{subsubsect:outerapp}
If $\mathcal{A}(\cdot)$ is a closed set, then by Remark~\ref{rem:characterization}, the set of all buy prices for a claim $C_T \in L(\mathcal{F}_T,\R^d)$ can be written as
\begin{align}
\begin{split}\label{eq:Pb_exactCVOP}
P^b(C_T) &= \{p \in \R^d \st V(x_0-p,C_T)\supseteq V(x_0,0)\} \\
&= \{p \in \R^d \st \forall w \in \R^q_+: \sup_{V_T\in \mathcal{A}(x_0-p)} w^TU(V_T+C_T)\geq v^w\},
\end{split}
\end{align}
where $v^w=\sup_{V_T\in \mathcal{A}(x_0)} w^TU(V_T)$.

Note that finding the values $v^w$ for all $w\in \R^q_+$ may not be possible in general. However, by the aforementioned approximation algorithms, we obtain a `representative' set ${W}$ of weight vectors. Then, clearly, $$P^b_{\text{out}}(C_T) :=  \{p \in \R^d \st \forall w \in {W}: \sup_{V_T\in \mathcal{A}(x_0-p)} ({w})^TU(V_T+C_T)\geq v^{w} \}$$
is a superset of $P^b(C_T)$. Moreover, $P^b_{\text{out}}(C_T)$ is the lower image of the following CVOP: 
\begin{align} \label{(P^b_out)}
&\text{maximize~~~~~} p \\
&\text{subject to~~~} ({w}^i)^T U(V_T^{i}+C_T) \geq v^{{w}^i};\notag\\
&\quad \quad \quad \quad \quad \:\; V_T^i\in \mathcal{A}(x_0-p) \text{~~ for~~} i=1,\ldots,l. \notag 
\end{align}
In general it is not known if this CVOP is bounded or not. In some cases, it is possible to formulate the problem as a bounded CVOP using an ordering cone different from $\R^d_+$. In Section~\ref{sect:Conical}, we consider a special case where the ordering cone is enlarged in order to solve problem~\eqref{(P^b_out)} using the algorithms provided in~\cite{cvop}.

{Using similar arguments one can show that the upper image of the following CVOP gives a superset $P^s_{\text{out}}(C_T)$ to $P^s(C_T)$:
	\begin{align} \label{(P^s_out)}
	&\text{minimize~~~~~} p \\
	&\text{subject to~~~} ({w}^i)^TU(V_T^i-C_T) \geq v^{{w}^i}; \notag\\
	&\quad \quad \quad \quad \quad \:\; V_T^i\in \mathcal{A}(x_0+p) \text{~~ for~~} i=1,\ldots,l. \notag  
	\end{align}}

\subsubsection{Computing a Subset of $P^b(C_T)$ and $P^s(C_T)$} \label{subsubsect:innerapp}
By Remark~\ref{rem:innerouterappr}, a finite weak $\epsilon$-solution $\bar{\mathcal{X}}=\{{X}^1, \ldots,{X}^l\}$ of~\eqref{(U)} provides an outer approximation of $V(x_0,0)$ given by $V_{\text{out}}(x_0,0):= \conv U(\bar{\mathcal{X}})-\R^q_++\epsilon \{k\}$, where $k \in \Int\R^q_+$ is fixed. Then,  $$P^b_{\text{in}}(C_T) :=  \{p \in \R^d \st \forall i=1,\ldots,l: \: \exists  V_T^i\in\mathcal{A}(x_0-p):\: U(V_T^i+C_T)\geq U({X}^i)+\epsilon k\}$$
is a subset of $P^b(C_T)$. To see that, let $p \in P^b_{\text{in}}(C_T)$, that is, for all  $i=1,\ldots,l$, there exist $V_T^i \in \mathcal{A}(x_0-p)$ such that $U(V_T^i+C_T)\geq U({X}^i)+\epsilon k$. Note that it is enough to show $V(x_0-p,C_T) \supseteq \conv U(\bar{\mathcal{X}})-\R^q_+ +\epsilon \{k\}$ as this implies $V(x_0-p,C_T)\supseteq V(x_0,0)$ and hence $p\in P^b(C_T)$. Let $\bar{u}\in\conv U(\bar{\mathcal{X}})$ be arbitrary. Then, there exist $\alpha_i\geq 0$ with $\sum_{i=1}^l\alpha_i = 1$ such that $\bar{u}=\sum_{i=1}^l\alpha_i U({X}^i)$. Note that $V_T^{\alpha} :=\sum_{i=1}^l\alpha_iV_T^i \in \mathcal{A}(x_0-p)$ by the convexity of $\mathcal{A}(x_0-p)$. Also, as the utility functional is concave we have $U(V_T^{\alpha}+C_T)\geq \sum_{i=1}^l\alpha_i U(V_T^i+C_T)$ and hence, $U(V_T^{\alpha}+C_T)\geq \bar{u}+\epsilon k$. Since for any $\bar{u}\in\conv U(\bar{\mathcal{X}})$, there exists $V_T^{\alpha}\in\mathcal{A}(x_0-p)$ such that $U(V_T^{\alpha}+C_T) \geq \bar{u}+\epsilon k$, $V(x_0-p,C_T) \supseteq \conv U(\bar{\mathcal{X}})-\R^q_+ +\epsilon \{k\}$ holds.  

$P^b_{\text{in}}(C_T)$ is the lower image of the following convex vector optimization problem: 
\begin{align} \label{(P^b_in)}
&\text{maximize~~~~~} p \\
&\text{subject to~~~} U(V_T^{i}+C_T) \geq U({X}^i)+\epsilon k;\notag\\
&\quad \quad \quad \quad \quad \:\; V_T^i\in \mathcal{A}(x_0-p) \text{~~ for~~} i=1,\ldots,l. \notag 
\end{align}

Using similar arguments one can show that the upper image $P^s_{\text{in}}(C_T)$ of the following CVOP is a subset of $P^s(C_T)$:
\begin{align} \label{(P^s_in)}
&\text{minimize~~~~~} p \\
&\text{subject to~~~} U(V_T^i-C_T) \geq U({X}^i)+\epsilon k; \notag\\
&\quad \quad \quad \quad \quad \:\; V_T^i\in \mathcal{A}(x_0+p) \text{~~ for~~} i=1,\ldots,l. \notag  
\end{align}
\begin{remark}
	It is possible that problems~\eqref{(P^b_in)} and~\eqref{(P^s_in)} {are} infeasible when the error bound $\epsilon$ in~\eqref{eq:innerouterappofupper} is not small enough, {see Example~\ref{ex:1}}. Thus, even though $P^b_{\text{in}}(C_T)$ and $P^s_{\text{in}}(C_T)$ are subsets of the set-valued buy and sell prices respectively, they could be empty sets. As it is not possible to determine the approximation error at this time, we do not call these sets outer or inner approximations, but rather sub- and supersets of $P^b(C_T)$ and $P^s(C_T)$. However, we will see that in the numerical examples of Sections~\ref{subsect:examples1} and~\ref{sect:Conical}, these sub- and supersets will approximate the set-valued prices rather well.
\end{remark}

\begin{remark} \label{rem:hedgepositions}
Note that solving the optimization problems~\eqref{(P^b_out)},~\eqref{(P^s_out)},~\eqref{(P^b_in)} and~\eqref{(P^s_in)}, one obtains a set of hedge positions $V_T^i, i=1,\ldots, l$. In practice, the decision maker could pick any of these {efficient} hedge positions as the vector valued expected utilities they provide are all maximal and they can not be compared with each other.  
\end{remark}

\subsubsection{Remarks on Computations in Some Special Cases}\label{subsubsect:remarkscomputation}
\begin{remark}\label{rem:indiff_d=1}
	For $d=1$,~\eqref{(P^b_out)},~\eqref{(P^s_out)},~\eqref{(P^b_in)} and~\eqref{(P^s_in)} are scalar convex programs. In this case, $P^b_{\text{out/in}}(C_T) = (-\infty,p^b_{\text{out/in}}]$ and $P^s_{\text{out/in}}(C_T)=[p^s_{\text{out/in}},\infty)$, where $p^b_{\text{out}}$, $p^s_{\text{out}}$, $p^b_{\text{in}}$ and $p^s_{\text{in}}$ are the optimal objective values of~\eqref{(P^b_out)},~\eqref{(P^s_out)},~\eqref{(P^b_in)} and~\eqref{(P^s_in)}, respectively. 
\end{remark}

\begin{remark}\label{rem:indiffprice_singlemulvar}
	For $d\geq 1$ and a complete preference relation which admits a single-prior single-utility representation (with utility function $u$), the set of buy prices $P^b(C_T)$ can be simplified to 
	\begin{align*}
	P^b(C_T) &= \{p \in \R^d ~ \st \sup_{V_T\in \mathcal{A}(x_0-p)} \E u(V_T+C_T)\geq v^0\},
	\end{align*}
	where  $v^0=\sup_{V_T\in \mathcal{A}(x_0)} \E u(V_T)$. Note that this is the lower image of the following convex vector optimization problem:
	\begin{align} \label{(Pb)}
	&\text{maximize~~~~~} p \\
	&\text{subject to~~~} \E u(V_T+C_T) \geq v^0  \notag\\
	&\quad \quad \quad \quad \quad \:\; V_T\in \mathcal{A}(x_0-p).  \notag
	\end{align}
	Similarly, $P^s(C_T)$ is the upper image of the following vector minimization problem 
	\begin{align*}
	&\text{minimize~~~~~} p \\
	&\text{subject to~~~} \E u(V_T-C_T) \geq v^0 \\
	&\quad \quad \quad \quad \quad \:\; V_T\in \mathcal{A}(x_0+p). 
	\end{align*}
	Thus, in the case of a complete preference relation and $d\geq 1$, it is not necessary to compute sub- and supersets of $P^b(C_T)$ and $P^s(C_T)$ as the set-valued prices $P^b(C_T)$ and $P^s(C_T)$ are upper respectively lower images of vector optimization problems itself.
\end{remark}
\section{Special Cases and Numerical Examples} \label{sect:appl}
We consider two different market models in this Section. The first one is an incomplete market where $d=1$ and the utility functions are univariate. In this setting, we consider an incomplete preference relation represented by multiple utility functions. The second one is the conical market model where $d>1$ and the utility functions are multivariate. Under this setting, we consider two different cases: a complete preference that is represented by a single multivariate utility function as in~\cite{benedetti2012}, and an incomplete preference relation represented by component-wise utility functions as in~\cite{Umax_setopt}. 
 
\subsection{An Example with Univariate Utility Functions}
\label{sect:univariatecase}
Consider a probability space $(\Omega, \mathcal{F}_T, \P)$ where ${\Omega} = \{\omega_j, j = 1, \ldots, 2^n\}$ and $\mathcal{F}_T = 2^{\Omega}$. Consider a single period model in a market consisting of one riskless and $n$ risky assets. The interest rate is assumed to be zero. Only $m<n$ of the risky assets can be traded. Assume the traded assets are indexed by $1,\ldots, m$. The current value of the traded and non-traded risky assets are $S^i_0$ for $i=1,\ldots,n$. At time $T$, the value of the traded and the non-traded assets are $S^i_T = S^i_0 \xi^i$, where $\xi^i, i=1, \ldots,n$ are $\mathcal{F}_T$ measurable random variables. Let $S_t$ be the vector of values of traded assets at time $t$, that is, $S_{t} = [S_{t}^1,\ldots,S_{t}^m]^T$ for $t = 0,T$. 

We consider a portfolio consisting of $\alpha \in \R^m$ shares of the traded assets and an amount $\beta = x_0 - \alpha^T S_0$ invested in the riskless asset, where $x_0$ is the initial endowment. Then, the wealth at the end of the period $[0,T]$ is given by $V_T = x_0 + \alpha^T(S_T-S_0)$. The set of wealth that can be generated with the initial endowment $x_0$ is $$\mathcal{A}(x_0) = \{V_T\in L^0(\mathcal{F}_T,\R) \st \exists \alpha \in \R^m \;:\: V_T \leq x_0 + \alpha^T(S_T-S_0)\},$$ which satisfies Assumption~\ref{assmp:A} a.-d. 

In this setting, we consider a claim (that may depend on the traded as well as on the non-traded assets), yielding a payoff $C_T$ at time $T$. We assume that there is a decision making committee consisting of $q$ individuals and the incomplete preference relation has a single-prior multi-utility representation. More precisely, assume that $\mathcal{Q}= \{\P\}$ and $\mathcal{U} = \{u_1,\ldots, u_q\}$ are such that Assumption~\ref{assump:utility} is satisfied. 

By Remark~~\ref{rem:CE_d=1}, the weak and the strong certainty equivalents of $C_T$ in this setting are $\Cw(C_T) = \{\cw\}$ and $\Cs(C_T) = \{\cs\}$ with
$$\cw = \inf_{i=1,\ldots,q}\{u_i^{-1}(\E u_i(C_T))\} \text{~~and~~~}\\
\cs = \sup_{i=1,\ldots,q}\{u_i^{-1}(\E u_i(C_T))\}.$$

Note that the market in consideration is incomplete, hence there is no unique complete market price. Instead, one could consider the no-arbitrage price bounds, which is nothing but the sub- and superhedging prices. However, these price bounds can be quite large for practical use, see also Remark~\ref{rem:SHP}. For the numerical example below, we compute both no-arbitrage price bounds and utility indifference price bounds to illustrate that the indifference price bounds provide a narrower interval.

In order to compute the indifference price bounds, we consider the utility maximization problem $P(x_0,0)$ in \eqref{(U)}, which can be formulated as $$\max_{\alpha \in \R^m} U(x_0 +\alpha^T(S_T-S_0)).$$ The set-valued buy and sell prices satisfy $\Int P^b(C_T) = (-\infty,p^b)$ and $\Int P^s(C_T) = (p^s,\infty)$, where $p^b$ and $p^s$ are the indifference price bounds. Note that as Assumption~\ref{assmp:A}~(e) may not be satisfied, one can not guarantee the closedness of the set-valued prices under this setting. The outer and inner approximations to the set-valued prices, $P^b_{\text{out/in}}(C_T)=(-\infty,p^b_{\text{out/in}}]$ and $P^s_{\text{out/in}}(C_T)=[p^s_{\text{out/in}},\infty)$, where $p^b_{\text{in}}\leq p^b \leq p^b_{\text{out}}$ and $p^s_{\text{out}} \leq p^s \leq p^s_{\text{in}}$, can be computed as it is explained in Remark~\ref{rem:indiff_d=1}. Below we provide a numerical example.

\label{subsect:examples1}
\begin{example}\label{ex:1} 
	Let $n=2, m=1, x_0 = 10, S_0 = [4 \:\: 6]^T$, $\P(\omega_i) = 0.25$ for $i=1,\ldots,4$ and
	\begin{align*}
	\xi^1(\omega_1)= \xi^1(\omega_2) = \frac52, \:\:\: \xi^1(\omega_3)=\xi^1(\omega_4)=\frac12, \:\:\: 
	\xi^2(\omega_1)= \xi^2(\omega_3) = \frac43, \:\:\: \xi^2(\omega_2)=\xi^2(\omega_4)=\frac23. 
	\end{align*}
	Assume that $\mathcal{U} = \{u_1,u_2\}$ where $u_1(x)=1-{e^{-x}},$ and {$u_2(x) = \log(\frac{x+10}{10})$} and let $C_T = \sum_{i=1}^2S_T^i$.
	First, as described above, we find the weak and the strong certainty equivalents of the payoff as $\cw=11.0889$ and $\cs=7.3678$, which also shows that the certainty equivalent as given in Definition~\ref{defn:CE} is empty. Then, we employ the dual algorithm proposed in~\cite{cvop} to obtain an approximation (with an error bound $\epsilon = 10^{-8}$) to the lower image of the utility maximization problem and the corresponding set of weight vectors $W$ as well as $v^w$ for each $w \in W$. The inner approximation of the lower image {$V(x_0,0)$} can be seen in Figure~\ref{ex1:V0}. When we solve the utility maximization problem, it also gives a subset of hedge positions that would yield maximal expected utilities. The expected utilities that can be generated by these hedge positions are marked on the boundary of the lower image $V(x_0,0)$.
	
	\begin{figure}[ht]
		\centering
		\includegraphics[width=11cm,height=8cm]{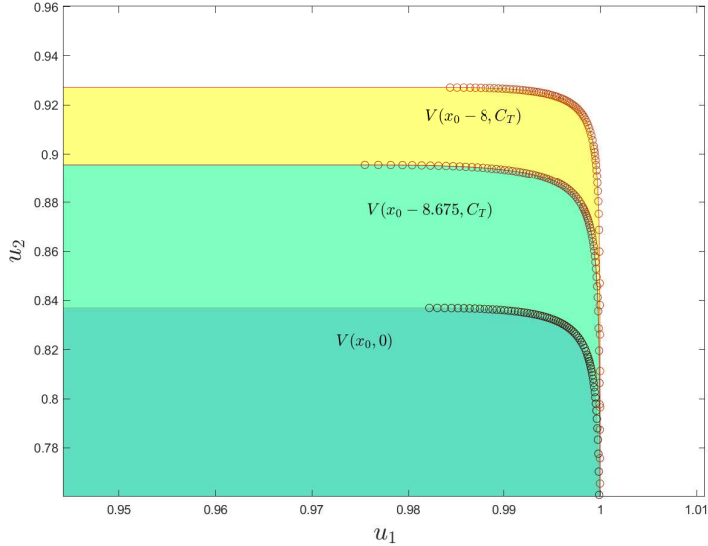}
		\caption {The inner approximation of the lower image $V(x_0,0)$ of problem $P(x_0,0)$ in \eqref{(U)} from Example~\ref{ex:1}.}
		\label{ex1:V0}
	\end{figure}
	
	After solving the utility maximization problem, we solve the single objective convex programs~\eqref{(P^b_out)}-\eqref{(P^s_in)} to compute $p^b_{\text{in}} = 8.6747$, $p^b_{\text{out}} = 8.6750$ and $p^s_{\text{out}} = 11.3250$, $p^s_{\text{in}}= 11.3253$. Similar to the utility maximization problem, when we solve these single optimization problems in order to find the price bounds, they also return a set of hedge positions which would yield maximal utilities that can be generated if the price is set accordingly. In Figure~\ref{ex1:V0}, we also plot the lower image $V(x_0-p^b_{\text{out}},C_T)$ of the corresponding utility maximization problem when the buy price is set to $p^b_{\text{out}} = 8.6750$. As expected, $V(x_0-p^b_{\text{out}},C_T) \supseteq V(x_0,0)$. The markers on the boundary of $V(x_0-p^b_{\text{out}},C_T)$ shows the expected utilities that can be generated by the hedge positions that are found by solving this utility maximization problem.
	
	For this example, the set of all superhedging and subhedging portfolios can be computed easily as $\SHP(C_T)=[12,\infty)$ and $\SubHP(C_T)=(-\infty,8]$, see Remark~\ref{rem:SHP}. In Figure~\ref{ex1:V0}, we plot the lower image $V(x_0-p_{\text{sub}},C_T)$ of the corresponding utility maximization problem when the buy price is set to the subhedging price, $p_{\text{sub}} = 8$.  As before the markers on the boundary of $V(x_0-p_{\text{sub}},C_T)$ shows the expected utilities that can be generated by the hedge positions that are found by solving this utility maximization problem. We observe that $V(x_0-p_{\text{sub}},C_T)\supseteq V(x_0-p^b_{\text{out}},C_T) \supseteq V(x_0,0)$ as expected. This illustrates that the buyer is still willing to buy the claim even if the price is higher than the subhedging price as the expected utility that he can generate is still greater than the expected utility that can be generated without buying the claim.
	
	We compute $p^b_{\text{out/in}}$ and  $p^s_{\text{out/in}}$ for different $\epsilon$ values. For this example we observe that the outer approximations are tight even for large $\epsilon$ values. Indeed, {it holds that} $p^b_{\text{out}}= 8.6750$ and $p^s_{\text{out}}= 11.3250$ for all $\epsilon$ values listed in Table~\ref{table1}. However, the inner approximations improve significantly as $\epsilon$ gets smaller. Below, we provide $p^b_{\text{in}}, p^s_{\text{in}}$ as well as the differences $p^b_{\text{out}}-p^b_{\text{in}}$ and $p^s_{\text{in}}- p^s_{\text{out}}$. Note that for large $\epsilon$, problems~\eqref{(P^b_in)} and~\eqref{(P^s_in)} turn out to be infeasible.
	\begin{table}[ht]
		\caption{Inner and Outer Approximations for $p^b$ and $p^s$ for Example~\ref{ex:1}}
		\centering
		\begin{tabular}{ |c|c|c|c|c| }
			\hline\hline
			& & & & \\ [-0.7ex]
			$\epsilon$ & $p^b_{\text{in}}$ & $p^s_{\text{in}}$ & $p^b_{\text{out}}-p^b_{\text{in}}$ & $p^s_{\text{in}}-p^s_{\text{out}}$ \\ [0.7ex]
			\hline\hline
			$10^{-4}$ & $-\infty$ & $\infty$ & $\infty$ & $\infty$  \\ 
			$10^{-5}$ & $8.3859$  & $11.6141$ & $0.2891$ & $0.2891$  \\ 
			$10^{-6}$ & $8.6496$  & $11.3504$ & $0.0254$ & $0.0254$  \\ 
			$10^{-7}$ & $8.6725$  & $11.3275$ & $0.0025$ & $0.0025$ \\ 
			$10^{-8}$ & $8.6747$  & $11.3253$ & $2.2134\times 10^{-4}$ & $2.3590\times 10^{-4}$ \\ \hline
		\end{tabular}
		\label{table1}
	\end{table}
\end{example}

\subsection{Conical Market Models and Multivariate Utility Functions}
\label{sect:Conical}
{In this Section, we consider conical market models, where we have $d>1$. In Section~\ref{subsect:complete} we study an example of a complete preference relation given by a multivariate utility function and in Section~\ref{subsect:componentwise} we consider an example of an incomplete preference relation.}

Throughout this Section, consider a financial market consisting of $d$ {currencies}, which can be traded over discrete time $t=0,1,..., T$. Let $(\Omega, \mathcal{F}, (\mathcal{F})_{t=0}^T, \P)$ be a filtered finite probability space. A portfolio vector at time $t$ is an $\mathcal{F}_t$ measurable random vector $V_t$, where the $i^{th}$ coordinate denotes the amount of money in currency $i$ at time $t$. Note that we do not fix a reference {currency as} a num\'{e}raire, instead all {currencies} are symmetrically treated.\\

For a market with proportional transaction costs, one models the market with a stochastic process $K_t$ of \emph{solvency cones}. A solvency cone $K_t$ is a polyhedral convex cone with $\R^d_+ \subsetneq K_t \ne \R^d$, and it denotes all positions in the $d$ currencies that can be traded to the zero portfolio by either exchanging or discarding currencies at time $t$. In other words, the generating vectors of $K_t$ are given by the bid-ask prices between any two currencies at time $t$. For this market, an $\R^d$-valued process, $(V_t)_{t=0}^T$ is called a \emph{self-financing} portfolio process, if $(V_t)$ is adapted and satisfies $$V_t - V_{t-1} \in -K_t \text{,~~~} \P \text{-} a.s. \text{,~~~for all~} t \in \{0,1,2,..., T\}$$ with $V_{-1}=0$. \\

We consider the linear space of all $\mathcal{F}_t$-measurable $\R^d$-valued random vectors $L_d^0(\mathcal{F}_t,\R^d)$. The set of all such vectors with values that are $\P$-$a.s.$ in $K_t$ is denoted by $L_d^0(\mathcal{F}_t, K_t)$. Furthermore, $A_T \subseteq L_d^0(\mathcal{F}_T, \R^n)$ denotes the set of all random vectors $V_T$, which are the values of a self-financing portfolio at time $T$. By definition of self financing processes, we have 
\begin{align} \label{A_T}
A_T = - L_d^0(\mathcal{F}_0, K_0)-L_d^0(\mathcal{F}_1, K_1) - \ldots - L_d^0(\mathcal{F}_T, K_T).
\end{align} 
Note that $A_T$ is the set of superhedgeable claims starting from initial endowment $0\in\R^d$ at time zero. 
Clearly, $\mathcal{A}(x_0):= x_0 + A_T$ is the set of all random vectors, which are the value of a self-financing portfolio at time $T$, where the initial endowment is $x_0 \in \R^d$ at time $t=0$.\\

\begin{remark}
	\label{rem:specialcase_assump}
	Note that $\mathcal{A}(\cdot)$ described above satisfies Assumption~\ref{assmp:A} as we will see in the following. Moreover, it satisfies a stronger monotonicity property given by 
	\begin{enumerate}
		\item[] \~{b}. If $x \leq_{K_0} y$, then $\mathcal{A}(x) \subseteq \mathcal{A}(y)$.
	\end{enumerate}
	To see that, let $x \leq_{K_0} y$, then
	\begin{align*}
	\mathcal{A}(x) &= x - L_d^0(\mathcal{F}_0, K_0)-L_d^0(\mathcal{F}_1, K_1) - \ldots - L_d^0(\mathcal{F}_T, K_T)\\
	&= y - (y-x) -L_d^0(\mathcal{F}_0, K_0)-L_d^0(\mathcal{F}_1, K_1) - \ldots - L_d^0(\mathcal{F}_T, K_T)\\
	& \subseteq y  - L_d^0(\mathcal{F}_0, K_0)-L_d^0(\mathcal{F}_1, K_1) - \ldots - L_d^0(\mathcal{F}_T, K_T)
	 = \mathcal{A}(y),
	\end{align*}
	where we used the fact that $y-x \in K_0$. 
	
	Clearly,  property~\~{b}. implies Assumption~\ref{assmp:A} b. as $\R^q_+ \subseteq K_0$. Also, using the convexity of $K_t$ and $L_d^0(\mathcal{F}_t,K_t)$, for $t =0,\ldots,T$, and by definition of $\mathcal{A}(\cdot)$, it is easy to see that Assumption~\ref{assmp:A} a., c. and d. hold. {Finally, if $A_T$ is closed, then Assumption~\ref{assmp:A} e. also holds. Note that $A_T$ is closed under the standard no arbitrage assumptions, see for instance~\cite{NA_conical}.} 
\end{remark}

\begin{remark} 
	\label{rem:strictprices_componntwise}
	For conical market models, in addition to Proposition~\ref{prop:lowerconvexset}, $P^b(C_T), P^s(C_T)$ satisfy also the following properties:
	\begin{enumerate}[a.]
		\item $P^b(C_T)$ and $P^s(C_T)$ are convex lower, respectively upper, sets with respect to $\leq_{K_0}$. 
		\item $P^b(\cdot)$ and $P^s(\cdot)$ are increasing with respect to the partial order $\leq$, in the sense of set orders $\curlyeqprec_{K_0}$ and $\preccurlyeq_{K_0}$, respectively: For  $C_T^1, C_T^2 \in L(\mathcal{F}_T, \R^d)$, if $C_T^1 \leq C_T^2$, then $P^b(C_T^1) \curlyeqprec_{K_0} P^b(C_T^2)$ and $P^s(C_T^1) \preccurlyeq_{K_0} P^s(C_T^2)$. 
	\end{enumerate}
	These can be shown using the fact that $\mathcal{A}(x)$ satisfies the additional property \~{b}. given in Remark~\ref{rem:specialcase_assump}.
\end{remark}

By Remark~\ref{rem:strictprices_componntwise} a., since $K_0 \supsetneq \R^d_+$, the optimization problems~\eqref{(P^b_out)},~\eqref{(P^s_out)} are not (and \eqref{(P^b_in)} and \eqref{(P^s_in)} may not be) bounded in the sense of vector optimization when the ordering cones of these problems is set as $\R^d_+$. Since the algorithms provided in~\cite{cvop} work only for bounded convex vector optimization problems, and since property ~\~{b}. and the properties in Remark~\ref{rem:strictprices_componntwise} are satisfied, one can set the ordering cones of these problems to be $K_0$ in a model with proportional transaction costs. In general, one still can not guarantee that these problems are bounded with respect to these extended ordering cones. However, the algorithms in~\cite{cvop} can determine if the problem is unbounded or bounded and solves it in case it is bounded. In the numerical examples considered below, the problems will turn out to be indeed bounded with respect  to $K_0$.

We will now consider two special cases in this market model. First, a complete preference relation represented by a single multivariate utility function is used. Then, we consider an incomplete preference relation represented by a single-prior multi-utility representation where the multivariate utility functions are defined component-wise. 

\subsubsection{A Single Multivariate Utility Function Case} 
\label{subsect:complete}
We consider a conical market model as described above and assume that the preference relation is complete and represented by a single multivariate utility function $u$, as discussed in Remark~\ref{rem:indiffprice_singlemulvar}. 

Indifference pricing with a multivariate utility function, where proportional transaction costs are modeled by solvency cones, has also been studied by Benedetti and Campi in~\cite{benedetti2012}. They consider a continuous time setting where the probability space is not necessarily finite. Accordingly, they have further assumptions on the multivariate utility function. The utility indifference buy price $p^b_j \in \R$ of a claim $C_T \in L(\mathcal{F_T},\R^d)$ in terms of currency $j \in \{1,\ldots,d\}$ is defined in~\cite{benedetti2012} as a solution to
\begin{equation}\label{eq:mulvar_scalardefn}
\sup_{V_T \in \mathcal{A}(x_0-e_jp^b_j)}\E{u}(V_T+C_T) = v^0,
\end{equation}
where $v^0 := \sup_{V_T \in \mathcal{A}(x_0)}\E{u}(V_T)$ and $e_j\in\R^d$ is the unit vector with only $j^{th}$ component being nonzero. It has been shown in~\cite{benedetti2012} that for all $j\in\{1,\ldots,d\}$, $p^b_j \in \R$ exists uniquely. 

We will now show that the utility indifference buy price $p^b_j$ defined in~\cite{benedetti2012} is contained {on the boundary of the set-valued buy price $\bd P^b(C_T)$} defined here. Thus, $p^b_j$ can be seen as a special case if one is only interested in the price expressed in currency $j$. However, the set-valued prices also allow for situations in which the buyer (or seller) has capital in several currencies, see also Remark~\ref{rem:scalarization_2} below. Then, it would be more expensive, if one would need to exchange that portfolio into a particular currency to buy the claim at price $p^b_j$ because of the transaction costs involved. Let us now show the relation between $p^b_j$ and $P^b(C_T)$.

First note that $p^b_j$ is the optimal objective value of the scalar convex program given by
\begin{align}\label{(p1)}
\nonumber&\text{maximize~~~~~} p \\
&\text{subject to~~~} \E {u}(V_T+C_T) \geq v^0 \\
\nonumber&\quad \quad \quad \quad \quad \:\; V_T\in \mathcal{A}(x_0-pe_j), 
\end{align}
which is equivalent to solving
\begin{align}\label{(p2)}
\nonumber&\text{maximize~~~~~} p_j \\
&\text{subject to~~~} \E {u}(V_T+C_T) \geq v^0 \\
\nonumber&\quad \quad \quad \quad \quad \:\; V_T\in \mathcal{A}(x_0-p)\\
\nonumber&\quad \quad \quad \quad \quad \:\; p_i\geq 0 \text{~~for~} i\neq j
\end{align}
in the following sense: If $p^1 \in \R, V_T^1\in\mathcal{A}(x_0-p^1e_j)$ is optimal for~\eqref{(p1)}, then $p=p^1 e_j \in \R^d, V_T^1 \in \mathcal{A}(x_0-p^1e_j)$ is optimal for~\eqref{(p2)}. On the other hand, if $p^2 \in \R^d, V_T\in\mathcal{A}(x_0-p^2)$ is optimal for~\eqref{(p2)}, then $p_j^2\in\R, V_T \in\mathcal{A}(x_0-p_j^2e_j)$ is optimal for~\eqref{(p1)}. Note that the epigraph form of~\eqref{(p2)} is 
\begin{align*}
&\text{maximize~~~~~} \rho \\ 
&\text{subject to~~~} \E {u}(V_T+C_T) \geq v^0 \\
&\quad \quad \quad \quad \quad \:\; V_T\in \mathcal{A}(x_0-p)\\
&\quad \quad \quad \quad \quad \:\; p \geq \rho e_j,
\end{align*}
which is the Pascoletti-Serafini scalarization of the vector optimization problem~\eqref{(Pb)} provided in Remark~\ref{rem:indiffprice_singlemulvar} with reference point $v = 0\in\R^d$ and direction $d = e_j$. Then, by Proposition~\ref{prop:scalarizationPS}, a solution of the scalarization problem is a weak minimizer for the vector optimization problem given by~\eqref{(Pb)}. Thus, $p^b_j$ corresponds to the point on the boundary of $P^b(C_T)$ that provides the utility indifference buy price in currency $j$.

The utility indifference sell price $p^s_j$ of $C_T$ in terms of currency $j$ is defined similarly and can in total analogy be computed by solving a convex program which is equivalent to a Pascoletti-Serafini scalarization of the corresponding vector problem. 

With these observations, we conclude that the set-valued buy and sell prices for multivariate utility functions as described in Remark~\ref{rem:indiffprice_singlemulvar} contain the real-valued utility indifference buy and sell prices in terms of a fixed currency as defined by Benedetti and Campi in~\cite{benedetti2012} in the sense that $p^b_je_j \in \bd P^b(C_T)$ and $p^s_je_j \in \bd P^s(C_T)$. 

\begin{remark}\label{rem:scalarization_1}
	There are many different scalarization approaches for vector optimization. The particular scalarization described above yield buy and sell prices in terms of a single currency. However, depending on the situation one could consider different scalarizations to compute a vector-valued price bound on the boundary of $P^b(\cdot)$ or $P^s(\cdot)$. Indeed, for practice it might be sufficient to obtain a single (or finitely many) vector-valued price bound(s) by solving single objective optimization problem(s) instead of solving a vector optimization problem.    	
\end{remark}
\begin{remark}\label{rem:scalarization_2}
	Remark~\ref{rem:scalarization_1} can even be extended in many different ways. For example, assume there is a potential buyer and a potential seller for a certain payoff $C_T$ with multivariate utility functions $u^b,u^s$, and initial endowment vectors $x^b,x^s$ in the $d$ currencies, respectively. In order to decide if there would be a trade between them, one could check if the set-valued buy price of the buyer and the set-valued sell price of the seller have a nonempty intersection. For this, one needs to compute $v^b := \sup_{V_T \in \mathcal{A}(x^b)}\E{u^b}(V_T)$ and $v^s := \sup_{V_T \in \mathcal{A}(x^s)}\E{u^s}(V_T)$ first. Then, the buy and sell prices are 
	\begin{align*}
	P^b(C_T) &= \{p \in \R^d ~ \st \sup_{V_T\in \mathcal{A}(x^b-p)} \E u^b(V_T+C_T)\geq v^b\},\\
	P^s(C_T) &= \{p \in \R^d ~ \st \sup_{V_T\in \mathcal{A}(x^s+p)} \E u^s(V_T-C_T)\geq v^s\}.
	\end{align*}
	In order to check if these sets have a nonempty intersection, one way is to minimize the distance between them by solving the following single objective problem:
	\begin{align*}
	\text{~minimize~~~~~~~~} & \norm{p^b-p^s}\\
	\text{subject to~~~~~~~~}  & \E u^b(V_T^b+C_T) \geq v^b, \\
	\text{~~~~~~~~~~~~~~~~~~~~~}  & \E u^s(V_T^s-C_T) \geq v^s,\\
	\text{~~~~~~~~~~~~~~~~~~~~~}  & V_T^b\in \mathcal{A}(x^b-p^b),\\
	\text{~~~~~~~~~~~~~~~~~~~~~}  & V_T^s\in \mathcal{A}(x^s+p^s).
	\end{align*}
	If the objective function value is zero, then the optimal solution yields a vector-valued buy/sell price $p^b=p^s$ as well as the trading strategies for the buyer and the seller. 
\end{remark}

\begin{example} \label{ex:multivarcomplete}
	Consider a financial market with $d=2$ currencies which can be traded over a single time period. The probability space at terminal time $T>0$ is given by $(\Omega,\mathcal{F}_T,\P)$ with $\Omega = \{\omega_1,\omega_2\}$, $\mathcal{F}_T = 2^{\Omega}$ and $p_i = \P(\omega_i) = \frac12$ for $i=1,2$. The generating vectors of the solvency cones $K_0, K_1(\omega_1)$ and $K_1(\omega_2)$ are given by the columns of the matrices $$K_0 = \left[ \begin{array}{cccc}
	1 & -0.9  \\
	-0.9 & 1  \end{array} \right],\:\: K_1^{(1)} = \left[ \begin{array}{cccc}
	2 & -1.9  \\
	-1 & 1  \end{array} \right] \text{~~and~~~} K_1^{(2)} = \left[ \begin{array}{cccc}
	1 & -1  \\
	-2 & 2.1  \end{array} \right],$$ respectively. Assume that the initial position is $x_0 = 0 \in \R^2$. We consider a payoff $C_T$ given by $C_T(\omega_1) = [1 \:\: 0]^T$, and $C_T(\omega_2) = [0 \:\: 1]^T$ and a multivariate utility function given by $u(x) = 1 - 0.5(e^{-x_1}+e^{-x_2})$.\\
	First, we compute the certainty equivalent of $C_T$ under the preference relation represented by the utility function $u$. This can be done as explained in Remark~\ref{rem:CEcompletemulvarU}, but for this example it is also possible to compute it analytically. Indeed, $\C(C_T)$ is nothing but the indifference curve for $u(x)$ computed at $u(x) = \E u(C_T) = \frac12(1-e^{-1})$. Also, in this example $\C(C_T)=\Cw(C_T)=\Cs(C_T)$. Figure~\ref{fig:CEmultivar} shows $\Cupp(C_T), \Clow(C_T)$ and $\C(C_T)$.\\ 
	\begin{figure}[ht]
		\centering
		\includegraphics[width=9cm,height=6cm]{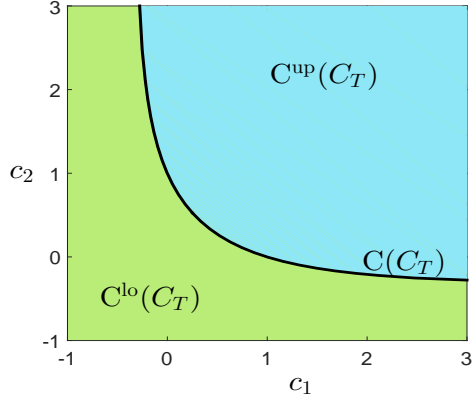}
		\caption {$\C(C_T)=\Cs(C_T)=\Cw(C_T)$ (black line), $\Cupp(C_T)$ (blue), $\Clow(C_T)$ (green) for Example~\ref{ex:multivarcomplete}.}
		\label{fig:CEmultivar}
	\end{figure}
	In order to compute the set-valued buy and sell prices, first we find $v^0$ as the optimal objective value of the utility maximization problem $P(x_0,0)$ given in \eqref{(U)}, which can be formulated as 
	\begin{align*}
	&\text{maximize~~} \sum_{i=1}^2 p_i u\left(x_0 - (K_0\alpha)^T - (K_1^{(i)}\beta_i)^T\right) \\
	&\text{subject to~~} \alpha, \beta_1, \beta_2 \in \R^2_+.
	\end{align*}
	Clearly, the feasible region is closed and the problem is bounded as the utility function is bounded. Then, as described in Remark~\ref{rem:indiffprice_singlemulvar}, we compute the set-valued buy and sell prices using the dual convex Benson algorithm from~\cite{cvop} with error bound $\epsilon = 10^{-5}$. In addition, we compute the set of all superhedging and subhedging portfolios, see Remark~\ref{rem:SHP}. Note that $\SubHP(C_T)$ and $\SHP(C_T)$ can be computed exactly by solving linear vector optimization problems, see~\cite{SHP}. The scalar buy and sell prices as defined in~\cite{benedetti2012} are also computed. Figure~\ref{fig:multivar} shows the set-valued buy and sell prices, the superhedging and subhedging portfolios and the scalar prices in terms of the corresponding currency. As it can be seen from the figure, the scalar buy and sell prices as in~\cite{benedetti2012} are points on the boundary of the buy and sell prices where one component is fixed at zero as expected. Moreover, $\SubHP(C_T)\subseteq P^b(C_T)$ and $\SHP(C_T)\subseteq P^s(C_T)$ as it was proven in Remark~\ref{rem:SHP}.
	\begin{figure}[ht]
		\centering
		\includegraphics[width=12.5cm,height=8cm]{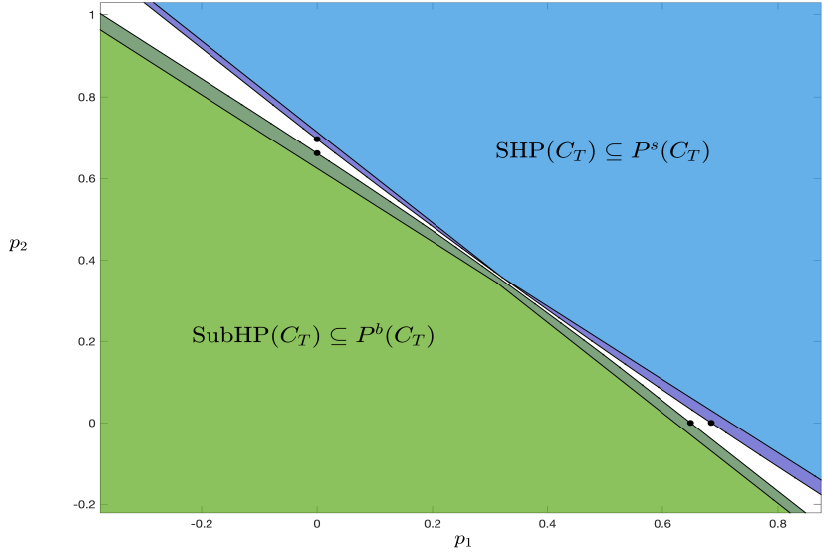}
		\caption {Set-valued buy price $P^b(C_T)$ (dark green) and sell price $P^s(C_T)$ (dark blue); set of all subhedging portfolios $\SubHP(C_T)$ (light green) and superhedging portfolios $\SHP(C_T)$ (light blue); scalar buy and sell prices $p^b_1=0.6487, p^b_2=0.6622, p^s_1=0.6846, p^s_2=0.6966$ (black marks) for Example~\ref{ex:multivarcomplete}.}
		\label{fig:multivar}
	\end{figure}
\end{example}

In order to illustrate the use of the proposed set valued prices compared to the scalar indifference prices as well as the sub/superhedging portfolios, we consider the following scenario. Assume that a possible buyer in this set up has the position $(0.2, 0.45)^T$ at time zero. It can be seen from Figure~\ref{fig:multivar} that this is in $P^b(C_T)$ but not in $\SubHP(C_T)$. Hence, this decision maker would not buy the claim if she considers the subhedging portfolios. However, she would buy it considering the utility that can be generated as buying the claim would yield more expected utility than not buying it. Moreover, if the prices are given in terms of a single currency as in \cite{benedetti2012}, then, in order to buy the claim, she needs to trade her position to the first or the second currency using the bid-ask prices provided at time zero ($K_0$). Notice that she can generate at most $0.605$ in the first currency or $0.63$ in the second currency at time zero. Since $0.605<0.6487 = p^b_1$ and $0.63 < 0.6622 = p^b_2$ the decision maker {can} not buy the claim {at} these scalar prices. In other words, buying the claim at a vector valued price, she is able to generate a higher utility than not buying it; however, if she {has to trade} at time zero in order to buy the claim at a scalar price, she can not anymore {exchange her position into sufficient capital} to buy it.
\subsubsection{Component-wise Utility Functions Case}
\label{subsect:componentwise}

Under the conical market model described above, we consider a component-wise utility representation as in~\cite{Umax_setopt}, where the utility function of each currency is considered separately. More specifically, we consider a single-prior multi-utility representation with $\mathcal{U} = \{\bar{u}^1,\ldots, \bar{u}^d\}$, where $\bar{u}^i:\R^d \rightarrow \R\cup\{-\infty\}$ is in the form of $\bar{u}^i(x) = u_i(x_i)$ for some univariate utility function $u_i : \R \rightarrow \R \cup \{- \infty\}$. Clearly, $u^i$ is increasing with respect to $\leq$ as $\bar{u}_i$ is increasing on its domain and $\bar{u}^i$ is proper concave as $u_i$ also is. Thus, $\bar{u}^i$ is a multivariate utility function. 

Under this setting, the sets $\Cupp(Z)$ and $\Clow(Z)$ for some $Z \in L^0(\mathcal{F},\R^d)$ simplify to
\begin{align*}
\Cupp(Z) &= \{c \in \R^d \st \forall i \in \{1,\ldots, d\}: u_i(c_i) \geq \E u_i(Z_i)\} =  \tilde{c} + \R^d_+,\\
\Clow(Z) &= \{c \in \R^d \st \forall i \in \{1,\ldots, d\}: u_i(c_i) \leq \E u_i(Z_i)\} = \tilde{c} - \R^d_+,
\end{align*} 
where $\tilde{c}:=(u_1^{-1}(\E u_1(Z_1)),\ldots, u_d^{-1}(\E u_d(Z_d)))$. Hence, the certainty equivalent of $Z$ is a singleton, namely, $\C(Z) = \{\tilde{c}\}$. Moreover, the strong and the weak certainty equivalents are the boundaries of $\Clow(Z)$ and $\Cupp(Z)$, respectively.

Even though the certainty equivalent has a much simpler form, the buy and sell prices do not necessarily simplify and one needs to compute the outer and inner sets as described in Sections~\ref{subsubsect:outerapp} and~\ref{subsubsect:innerapp}.

Below we provide an illustrative numerical example.

\begin{example}
	\label{ex:comp2dim} 
	Consider the same market model and payoff $C_T$ as in Example~\ref{ex:multivarcomplete}. The scalar utility functions are given by $u_i(x_i) = 1 - e^{-x_i}, x_i \geq 0$ for $i=1,2$. The certainty equivalent of $C_T$ under this utility representation is computed as $(0.3799, \: 0.3799)^T \in \R^2$. Figure~\ref{fig:CEmultivarcompwise} shows $\Cupp(C_T), \Clow(C_T), \Cw(C_T), \Cs(C_T)$ and $\C(C_T)$.
	
	\begin{figure}[ht]
		\centering
		\includegraphics[width=9cm,height=7cm]{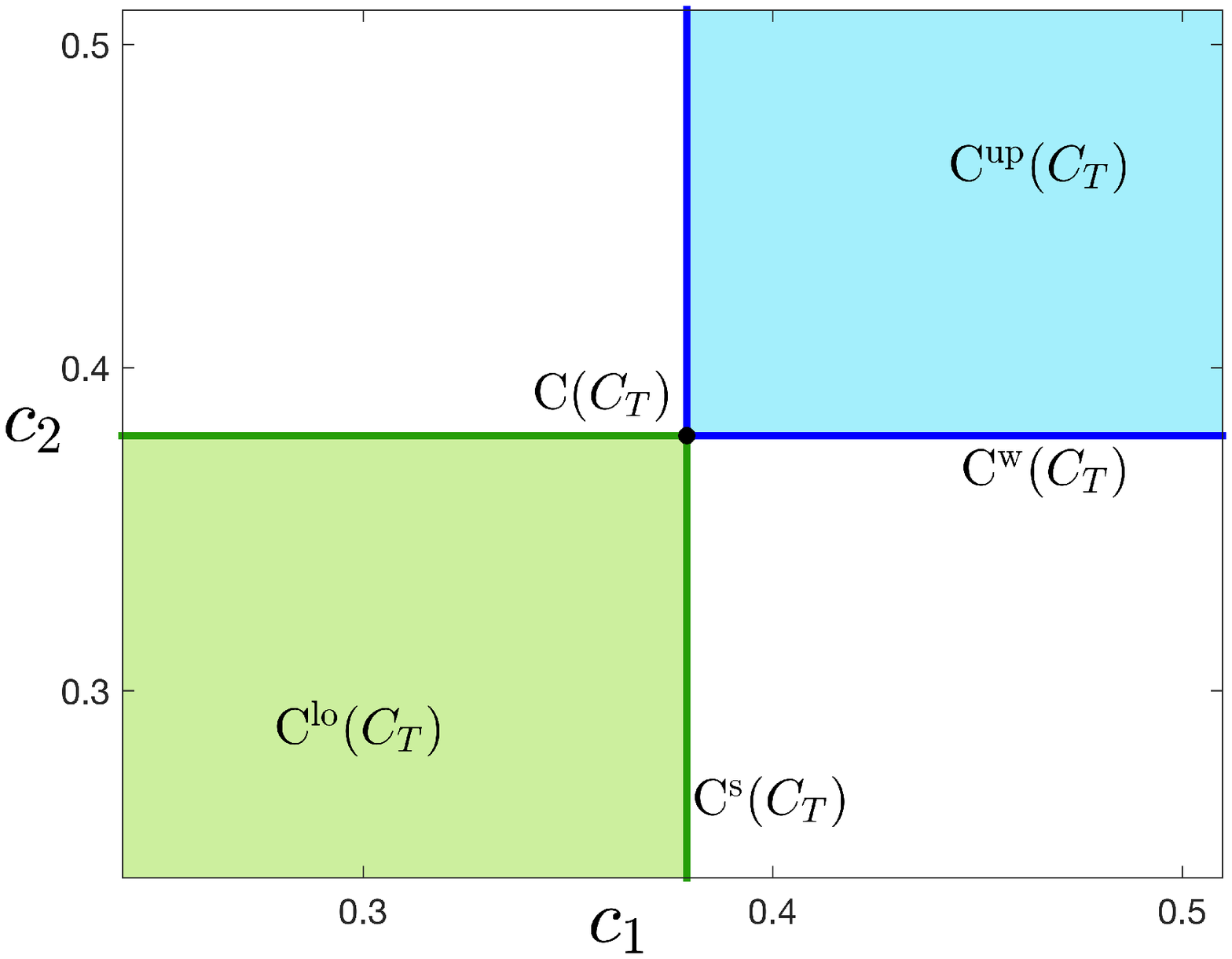}
		\caption {$\C(C_T)$ (black mark), $\Cupp(C_T)$ (blue), $\Clow(C_T)$ (green), $\Cs(C_T)$ (dark green line), and $\Cw(C_T)$ (dark blue line) for Example~\ref{ex:comp2dim}.}
		\label{fig:CEmultivarcompwise}
	\end{figure}
	
	Denoting $u=[u_1\; u_2]^T$, the utility maximization problem~$P(x_0,0)$ in \eqref{(U)} can be written as
	\begin{align*}
	&\text{maximize~~} \sum_{i=1}^2 p_i u\left(x_0 - (K_0\alpha)^T - (K_1^{(i)}\beta_i)^T\right) \\
	&\text{subject to~~} \alpha, \beta_1, \beta_2 \in \R^2_+.
	\end{align*}
	
	Clearly, the feasible region is closed and the problem is bounded as the utility functions are bounded. Hence, the indifference price bounds can be computed as explained in Section~\ref{subsubsect:strictcompute}. The error bound is taken as $\epsilon = 10^{-4}$. In Figure~\ref{fig:multivar_2} one can see $P^b(C_T)$ and $P^s(C_T)$ as well as $\SHP(C_T)$ and $\SubHP(C_T)$, which are the same as in Example~\ref{ex:multivarcomplete}. Note that we compute the supersets and subsets of $P^b(C_T)$ and $P^s(C_T)$. It is not possible to distinguish the boundaries of the inner and outer sets in the figure as they are very close to each other.
	
	\begin{figure}[ht]
		\centering
		\includegraphics[width=15.1cm,height=8.2cm]{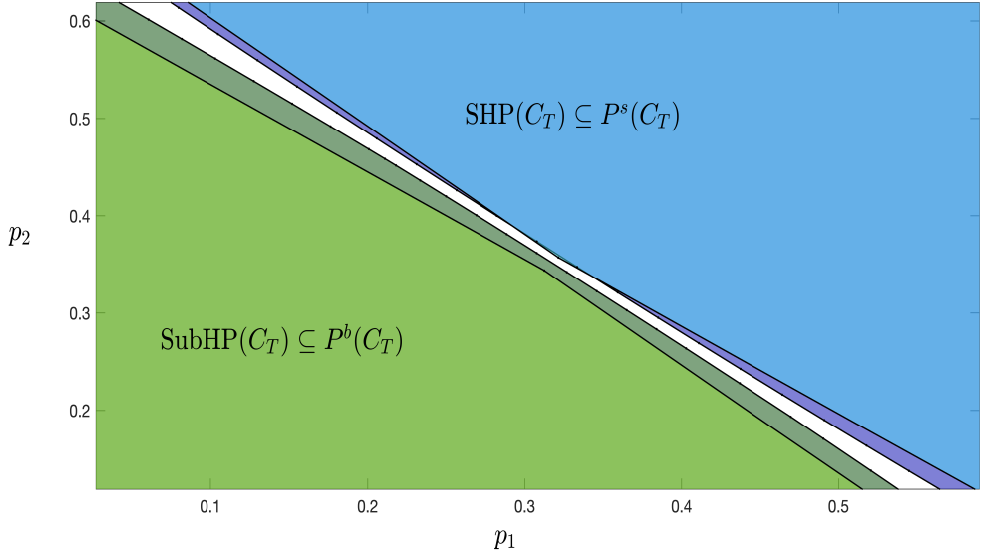}
		\caption {Set-valued buy price $P^b(C_T)$ (dark green) and sell price $P^s(C_T)$ (dark blue); set of all subhedging portfolios $\SubHP(C_T)$ (light green) and superhedging portfolios $\SHP(C_T)$ (light blue) for Example~\ref{ex:comp2dim}.}
		\label{fig:multivar_2}
	\end{figure}
\end{example}

\appendix

\section{Appendix: Proof of the results from Section~\ref{sect:indifference}}
\label{sect:App}
\begin{proof}[Proof of Proposition~\ref{prop:lowerconvexset}]
	\begin{enumerate}
		\item 
		We first show that $P^b(C_T)$ is a lower set, that is, $P^b(C_T) = P^b(C_T) - \R^d_+$. It is clear that $P^b(C_T) \subseteq P^b(C_T) - \R^d_+$. To show the reverse inclusion, let $p^b \in P^b(C_T)$ and $r \in \R^d_+$. By Assumption~\ref{assmp:A} c., $\mathcal{A}(x_0 - p^b +r) \supseteq \mathcal{A}(x_0-p^b)$. Then, by definition of $V(\cdot, \cdot)$ and since $p^b \in P^b(C_T)$, we have $V(x_0-p^p+r, C_T) \supseteq V(x_0-p^b, C_T) \supseteq V(x_0, 0)$. Thus, $p^b - r \in P^b(C_T)$. 
		
		Next we show that $P^b(C_T) \subseteq \R^d$ is a convex set. Let $p^1, p^2 \in P^b(C_T)$, $\lambda \in [0,1]$ and $p^{\lambda}:=\lambda p^1 + (1-\lambda)p^2$. First, note that $V(x_0-p^i, C_T) \supseteq V(x_0, 0)$ for $i=1,2$ implies $$\lambda V(x_0-p^1, C_T) + (1-\lambda)V(x_0-p^2, C_T) \supseteq V(x_0,0).$$ In order to show $p^{\lambda} \in P^b(C_T)$, it would be enough to prove
		\begin{equation}
		\label{eq:convexset1}
		V(x_0-p^{\lambda}, C_T) \supseteq \lambda V(x_0-p^1, C_T) + (1-\lambda)V(x_0-p^2, C_T).
		\end{equation} 
		Consider $U(V_T^i)-r^i \in V(x_0-p^i, C_T)$, where $V_T^i \in \mathcal{A}(x_0-p^i)$ and $r^i \in \R^q_+$ for $i =1,2$. By concavity of $u \in \mathcal{U}$, we have 
		\begin{align*}
		\lambda U(V_T^1+C_T)+ (1-\lambda)U(V_T^2+C_T) &\leq U(\lambda V_T^1 + (1-\lambda)V_T^2 + C_T). 
		\end{align*}
		Clearly, for some $\tilde{r}\in \R^q_+$, where $r^{\lambda}= \lambda r^1 + (1-\lambda)r^2$, it holds
		\begin{align*}
		\lambda U(V_T^1+C_T)+ (1-\lambda)U(V_T^2+C_T) -r^{\lambda} &= U(\lambda V_T^1 + (1-\lambda)V_T^2 + C_T) -r^{\lambda}- \tilde{r}.
		\end{align*}
		As $\lambda V_T^1 + (1-\lambda)V_T^2 \in \mathcal{A}(x_0 - p^{\lambda})$ by Assumption~\ref{assmp:A} b., we have
		\begin{align*}
		U(\lambda V_T^1 + (1-\lambda)V_T^2 + C_T) -r^{\lambda}- \tilde{r} \in V(x_0-p^{\lambda}, C_T),
		\end{align*}
		which implies~\eqref{eq:convexset1}.
		
		\item 	We show that $P^b(\cdot)$ is monotone with respect to $\leq$ and $\curlyeqprec$.
		Let $C_T^1, C_T^2 \in L(\mathcal{F}, \R^d)$ with $C_T^1 \leq C_T^2$ and $p^b \in P^b(C_T^1)$. Then, by Remark~\ref{prop:V_properties} a., and by the definition of $P^b(\cdot)$,
		$$V(x_0-p^b, C_T^2) \supseteq V(x_0-p^b, C_T^1) \supseteq V(x_0,0).$$ 
		This concludes that $P^b(C_T^1)\subseteq P^b(C_T^2)$, which implies $P^b(C_T^1) \curlyeqprec P^b(C_T^2)$. 
		
		\item
		We show $ \lambda P^b(C_T^1) + (1-\lambda)P^b(C_T^2) \subseteq P^b(C_T^{\lambda})$, which implies~\eqref{eq:Pbconcave}. Let $p^i \in P^b(C_T^i)$, that is, $V(x_0, 0) \subseteq V(x_0-p^i, C^i_T)$ for $i = 1,2$. Clearly, 
		\begin{equation*}
		V(x_0, 0) \subseteq \lambda V(x_0-p^1, C_T^1) + (1-\lambda)V(x_0-p^2, C_T^2).
		\end{equation*}
		Then, it would be enough to show that 
		\begin{equation}
		\label{eq:conc_eq1}
		\lambda V(x_0-p^1, C_T^1) + (1-\lambda)V(x_0-p^2, C_T^2) \subseteq V(x_0 - p^{\lambda}, C_T^{\lambda}),
		\end{equation} where $p^{\lambda}:= \lambda p^1 + (1-\lambda)p^2$. 
		Let $V_T^i \in \mathcal{A}(x_0-p^i)$ and $r^i \in \R^q_+$ for $i=1,2$. 
		By the concavity of $U(\cdot)$, we have
		\begin{align*}
		\lambda (U(V_T^1+C_T^1)-r^1) & + (1-\lambda) (U(V_T^2+C_T^2)-r^2) \\
		& = \lambda U(V_T^1+C_T^1) +  (1-\lambda) U(V_T^2+C_T^2) - r^{\lambda} \\
		& = U(\lambda V_T^1 + (1-\lambda)V_T^2 + C_T^{\lambda}) -r^{\lambda} -\tilde{r}
		\end{align*}
		for some $\tilde{r} \in \R^q_+$, where $r^{\lambda} := \lambda r^1 + (1-\lambda)r^2$. Note that $\lambda V_T^1 + (1-\lambda)V_T^2 \in \mathcal{A}(x_0-p^{\lambda})$ by Assumption~\ref{assmp:A} b. Then,~\eqref{eq:conc_eq1} is implied by
		\begin{equation*}
		\lambda (U(V_T^1+C_T^1)-r^1) + (1-\lambda) (U(V_T^2+C_T^2)-r^2) \in V(x_0 - p^{\lambda}, C^T_{\lambda}). 
		\end{equation*}
	\end{enumerate}
\end{proof}

\begin{proof}[Proof of Proposition~\ref{prop:intersectionPbPs}]
	\begin{enumerate}
		\item 
		Let $p \in P^b(C_T)\cap P^s(C_T)$, that is, $V(x_0-p,C_T) \supseteq V(x_0,0)$ and $V(x_0+p,-C_T) \supseteq V(x_0,0)$. 
		Let $v^b \in \bigcup_{V_T^b\in\mathcal{A}(x_0-p)}[U(V_T+C_T)-\R^q_+]$ and $v^s \in \bigcup_{V_T^b\in\mathcal{A}(x_0+p)}[U(V_T-C_T)-\R^q_+]$. Then, $v^b = U(V_T^b+C_T)-r^b$ and $v^s = U(V_T^s-C_T)-r^s$ for some $V_T^b\in \mathcal{A}(x_0-p)$, $V_T^s\in\mathcal{A}(x_0+p)$ and $r^b, r^s \in \R^q_+$. By the concavity of $U$, we have
		\begin{align*}
		\frac12 v^b+\frac12 v^s \leq U(\frac12 V_T^b+\frac12 V_T^s)-\frac12(r^b+r^s). 
		\end{align*}
		Note that for any $V_T^b\in\mathcal{A}(x_0-p)$ and $V_T^s\in\mathcal{A}(x_0+p)$, we have $V_T^b + p, V_T^s - p \in \mathcal{A}(x_0)$ and $\frac12 V_T^b + \frac12 V_T^s \in \mathcal{A}(x_0)$ by Assumption~\ref{assmp:A} d.
		Then, $U(\frac12 V_T^b+\frac12 V_T^s)-\frac12(r^b+r^s) \in V(x_0,0)$ and hence, $\frac12 v^b+\frac12 v^s \in V(x_0,0)$. \\ Now, for any $V^b \in V(x_0-p,C_T), V^s \in V(x_0+p,-C_T)$, there exists sequences $(v_n^b)\in \bigcup_{V_T^b\in\mathcal{A}(x_0-p)}[U(V_T+C_T)-\R^q_+], (v_m^s)\in\bigcup_{V_T^b\in\mathcal{A}(x_0+p)}[U(V_T-C_T)-\R^q_+]$ with $\lim_{n\rightarrow\infty} v^b_n = V^b$ and $\lim_{n\rightarrow\infty} v^s_n = V^s$. Since for each $n,m \geq 1$, $\frac12 v_n^b+\frac12 v_m^s \in V(x_0,0)$ and since $V(x_0,0)$ is closed, we have $\frac12 V^b+\frac12 V^s \in V(x_0,0)$, which proves that
		\begin{equation}\label{eq:equalimages}
		\frac12 V(x_0-p,C_T)+\frac12 V(x_0+p,-C_T) \subseteq V(x_0,0).
		\end{equation}
		As $V(x_0-p,C_T),V(x_0+p,-C_T),V(x_0,0)$ are lower closed sets and both $V(x_0-p,C_T), V(x_0+p,-C_T)$ are supersets of $V(x_0,0)$,~\eqref{eq:equalimages} holds only if we have $V(x_0-p,C_T) = V(x_0,0) = V(x_0+p,-C_T)$. 
		\item
		
		Let $p\in P^b(C_T)\cap P^s(C_T)$. Then, by Proposition~\ref{prop:intersectionPbPs}-1., $V(x_0-p,C_T)=V(x_0,0)$. Moreover, by Assumption~\ref{assmp:A} d., and by Remark~\ref{prop:V_properties} d., $V(x_0-p+\epsilon,C_T)\supsetneq~V(x_0,0)$ for any $\epsilon \in \Int \R^d_+$. Thus, $p - \epsilon \notin P^b(C_T)\cap P^s(C_T)$ for any $\epsilon \in \Int \R^d_+$. Then, $\Int (P^b(C_T)\cap P^s(C_T)) = \emptyset$. 
	\end{enumerate}
\end{proof}

\begin{proof}[Proof of Proposition~\ref{prop:closedness}]
	We will show that $P^b(C_T)$ is closed. By Proposition~\ref{prop:lowerconvexset}, it is enough to show that $\bd P^b(C_T) \subseteq P^b(C_T)$. Moreover, as $P^b(C_T)$ is a lower set for any $p \in \bd P^b(C_T)$, there exists a sequence $(p^n)_{n\geq 1} \in P^b(C_T)$ such that $p^{n+1} \geq p^{n}$ for all $n\geq 1$ and $\lim_{n \rightarrow \infty} p^n = p$. The proof will be complete if 
	\begin{equation}\label{eq:inclusion}
	V(x_0-p,C_T) \supseteq \bigcap_{n\geq1}V(x_0-p^n,C_T)
	\end{equation} 
	holds. Indeed, together with the fact that $V(x_0-p^n,C_T)\supseteq V(x_0,0)$ for all $n \geq 1$,~\eqref{eq:inclusion} implies that $V(x_0-p,C_T) \supseteq \bigcap_{n\geq1} V(x_0-p^n,C_T) \supseteq V(x_0,0)$; hence $p \in P^b(C_T)$.
	
	In order to show~\eqref{eq:inclusion}, first note that $\mathcal{A}(x_0-p)= \bigcap_{n\geq 1} \mathcal{A}(x_0-p^n)$ by Assumption~\ref{assmp:A} e. Then, we have 
	\begin{align*}
	V(x_0-p,C_T) &= \cl \bigcup_{V_T \in \mathcal{A}(x_0-p)} \left(U(V_T+C_T)-\R^q_+ \right)\\
	&=\cl \{U(V_T+C_T)-r \st \forall n \geq 1: V_T \in \mathcal{A}(x_0-p^n), r \in \R^q_+\}\\
	&=\cl \bigcap_{n\geq 1}\bigcup_{V_T \in \mathcal{A}(x_0-p^n)}\left(U(V_T+C_T) - \R^q_+\right).
	\end{align*}	
	Let $y \in \bigcap_{n \geq 1}V(x_0-p^n,C_T)$. Since $y$ is an element of the lower image $V(x_0-p^n,C_T)$, it is true that $\{y\} - \Int\R^q_+ \subseteq \bigcup_{V_T \in \mathcal{A}(x_0-p^n)}\left(U(V_T+C_T) - \R^q_+\right)$ for all $n \geq 1$. Let $(y^k)_{k\geq1} \in \{y\} - \Int\R^q_+$ be a sequence with $\lim_{k \rightarrow \infty} y^k = y.$ Clearly, for each $k\geq 1$, \\$y^k \in \bigcap_{n\geq 1}\bigcup_{V_T \in \mathcal{A}(x_0-p^n)}\left(U(V_T+C_T) - \R^q_+\right)$, hence $$y \in \cl \bigcap_{n\geq 1}\bigcup_{V_T \in \mathcal{A}(x_0-p^n)}\left(U(V_T+C_T) - \R^q_+\right).$$  
\end{proof}

\begin{proof}[Proof of Proposition~\ref{prop:uniformcont}]
	{Assume for a proof by contradiction that $V(x_0,0)\subseteq \Int V(x_0-p,C_T)$. Hence, there exists $\epsilon > 0$ such that $V(x_0-p,C_T)\supseteq V(x_0,0)+B(0,\epsilon)$. It is enough to show that there is $\tilde{p}>p$ such that 
		\begin{equation}\label{eq:ptilde}
		V(x_0-\tilde{p},C_T)+B(0,\epsilon)\supseteq V(x_0 - p,C_T) 
		\end{equation}
		as this implies that $V(x_0-\tilde{p},C_T) \supseteq V(x_0,0)$, hence $\tilde{p} \in P^b(C_T)$, which contradicts that $p\in\bd P^b(C_T)$.
		To prove \eqref{eq:ptilde} note that as each $u\in\mathcal{U}$ is uniformly continuous, there exists $\delta> 0$ such that $\norm{x-y}\leq \delta$ implies that $\abs{u(x)-u(y)}\leq \frac{\epsilon}{\sqrt{q}}$. Let $\tilde{\delta} \in (0,\delta)$ and $\tilde{p}:= p+\tilde{\delta}\frac{e}{\norm{e}}$, where $e\in \R^d$ is the vector of ones. Note that \begin{align*}
		V(x_0-\tilde{p},C_T) +B(0,\epsilon) &= \bigcup_{V_T\in\mathcal{A}(x_0-\tilde{p})}U(V_T+C_T)+B(0,\epsilon)-\R^q_+ \\
		&= \bigcup_{V_T\in\mathcal{A}(x_0-{p})}U(V_T+C_T-\tilde{\delta} \frac{e}{\norm{e}}) + B(0,\epsilon)-\R^q_+,
		\end{align*} 
		as $\mathcal{A}(x_0-\tilde{p}) = \mathcal{A}(x_0-p)-\tilde{\delta}\frac{e}{\norm{e}}$. Let $U(V_T+C_T)-r \in V(x_0-p,C_T)$ for some $V_T \in \mathcal{A}(x_0-p)$ and $r\in \R^q_+$. Clearly, for each $u\in\mathcal{U}$ and $Q\in\mathcal{Q}$, we have $$\absgg {\E_{Q}[u(V_T+C_T)]-\E_{Q}[u(V_T+C_T-\tilde{\delta}\frac{e}{\norm{e}})]}\leq \frac{\epsilon}{\sqrt{q}}. $$ 
		Then 
		\begin{align*}
		&\norm{U(V_T+C_T)-U(V_T+C_T-\tilde{\delta}\frac{e}{\norm{e}})}\\
		&~~~~~~~ = \bigg(\sum_{j=1}^s\sum_{i=1}^r \absgg  {\E_{Q_j}(u^i(V_T+C_T)-u^i(V_T+C_T-\tilde{\delta}\frac{e}{\norm{e}}))}^2\bigg)^{\frac{1}{2}}\leq \epsilon.\end{align*}Hence, $U(V_T+C_T) -r \in U(V_T+C_T-\tilde{\delta}\frac{e}{\norm{e}})+B(0,\epsilon)-\R^q_+$ and \eqref{eq:ptilde} holds.}
\end{proof}

\section*{Acknowledgment}
We are grateful to Andreas Hamel for initiating many applications of set-valued analysis and in particular introducing set-valued extensions of concepts from economics and financial mathematics, which shaped and motivated the subject of this manuscript.  In particular, he initiated the discussion of a set-valued certainty equivalent which resulted in~\cite{juniorproj}. We would also like to thank \"Ozg\"ur Evren, Zachary Feinstein, Efe Ok and Frank Riedel for fruitful discussions and their constructive feedback. {Furthermore, we would like to thank two anonymous referees for detailed and helpful comments and pointing us to further related literature.}

\bibliographystyle{plain}
\bibliography{reportbib2}

\end{document}